\documentclass[a4paper, 10pt]{article}
\usepackage{anysize}
\marginsize{3,5cm}{2,5cm}{2,5cm}{2,5cm}
\usepackage{amssymb}
\usepackage{amsmath,amsbsy,amssymb,amscd}
\usepackage{t1enc}
\usepackage[cp1250]{inputenc}
\usepackage[british]{babel}
\usepackage[all]{xy}
\usepackage{color}
\usepackage{amsfonts}
\usepackage{latexsym}
\usepackage{amsthm}
\usepackage{mathrsfs}

\usepackage{mathrsfs} 

\DeclareMathAlphabet{\mathpzc}{OT1}{pzc}{L}{it} 

\newtheorem{definition}{Definition}[section]

\newtheorem{proposition}[definition]{Proposition}
\newtheorem{theorem}{Theorem}
\newtheorem{corollary}[definition]{Corollary}
\newtheorem{remark}[definition]{Remark}
\newtheorem{lemma}[definition]{Lemma}
\newtheorem{lemmata}[definition]{Lemmata}

\def\geq{\geqslant}
\def\leq{\leqslant}
\def\R{\mathbb{R}}
\def\T{\mathbb{T}}
\def\eps{\varepsilon}

\def\Z{\mathbb{Z}}
\def\N{\mathbb{N}}

\def\Q{\mathbb{Q}}
\newcommand{\bea}{\begin{eqnarray}}
  \newcommand{\eea}{\end{eqnarray}}
  \newcommand{\beab}{\begin{eqnarray*}}
  \newcommand{\eeab}{\end{eqnarray*}}
\renewcommand{\a}{\alpha}
  \newcommand{\be}{\begin{equation}}
  \newcommand{\ee}{\end{equation}}
  \newcommand{\RQ}{\R \setminus \Q}
\title{Multiple mixing for a class of conservative surface flows}
\author{Bassam Fayad and Adam Kanigowski}
\begin{document}
\baselineskip=14pt \maketitle

\newcommand{\cD}{\mathcal D}
\newcommand{\cE}{\mathcal E}

\newcommand{\carre}{\hfill $\Box$}

\begin{abstract} Arnol'd and Kochergin mixing conservative flows on surfaces stand as the main and almost only natural class of mixing transformations for which higher order mixing has not been established, nor disproved. 

Under suitable arithmetic conditions on their unique rotation vector, of full Lebesgue measure in the first case and of full Hausdorff dimension in the second, we show that these flows are mixing of any order. 
For this, we show that they display a generalization of the so called Ratner property on  slow divergence of nearby orbits, that implies strong restrictions on their joinings, that in turn yield higher order mixing. 

This is the first case in which the Ratner property is used to prove multiple mixing outside its original context of horocycle flows and we expect our approach will have further applications. 

\end{abstract}


\section{Introduction} A major open problem in ergodic theory is Rokhlin's question on whether mixing implies mixing of all orders, also called multiple mixing \cite{Ro}. In most of the known examples of mixing dynamical systems, multiple mixing is now known to hold. Moreover, a positive answer to Rokhlin's question is actually known to generally hold  within various classes of mixing dynamical systems. The most noteworthy are K-systems where multiple mixing always holds \cite{Co-Fo-Si},  horocycle flows \cite{Marcus}, mixing systems with singular spectrum that display multiple mixing by a celebrated theorem of Host \cite{Ho}, and finite rank systems since  Kalikow showed that rank one and mixing implies multiple mixing \cite{Kal}, a result that was extended to finite rank mixing systems by Ryzhikov \cite{Ryz}.

 In the second half of the last century, Arnol'd and Kochergin introduced  a  major class of conservative smooth mixing flows on surfaces, with non-degenerate saddle type singularities for the first  and degenerate ones for the second. Mixing of these flows was proved by Kochergin in the case of degenerate power like singularities (see exact description below) in \cite{Koc2}  and by Khanin and Sinai in a particular case of non-degenerate asymmetric saddle type singularities (see exact description below) \cite{SK}.  The study of the mixing properties of these flows has known a revival of interest since the beginning of the 2000's, with results such as the computation of the speed of mixing \cite{Fa} or extension of the Kahnin-Sinai mixing result to include all irrational translation vectors \cite{Koc7} (see also \cite{Koc4},\cite{Koc5}), or advances in the study of Arnol'd and Kochergin flows in the general case where the Poincar\'e section return map is an interval exchange and not just a circular rotation \cite{CU1,CU2,CU3}.

  Arnol'd and Kochergin flows stand today as the main and almost only natural class of mixing transformations for which higher order mixing has not been established, nor disproved. Our aim here is to prove mixing of all orders for a rich subclass of these systems determined by the arithmetics of their unique rotation vector. For this, we use the representation of Arnol'd and Kochergin  flows as special flows above an irrational rotation $R_\a$  of the circle  and under a ceiling function with asymmetric logarithmic singularities for the first, and integrable power like singularities for the second. 
     Loosely speaking our main result is as follows (it will be made precise at the end of this introduction, see Corollaries \ref{ln} and \ref{pow}). 
 
 \medskip 
 
\noindent  {\bf Theorem.} {\it Arnol'd-Khanin-Sinai flows are mixing of all orders for a  set of $\a \in (0,1)$ of full Lebesgue measure. Kochergin flows are  mixing of all orders for a  set of $\a \in (0,1)$ of full Hausdorff dimension. }

\medskip

 Similar mixing mechanisms due to orbit shear as in Kochergin and Arnol'd-Khanin-Sinai flows were observed 
   relatively recently such as in \cite{Fa2} or \cite{AFU} and it should be possible to apply the techniques  of the current paper to the study of higher order mixing for such {\it  parabolic }   systems.

 To explain our approach we need first to make a detour by Ratner's study of horocycle flows. In the 1980's, M.\ Ratner developed a rich machinery  to study horocycle flows \cite{Rat}-\cite{Rat3} and, in particular, singled out a special way of controlled slow divergence of orbits of nearby points which resulted in the notion of $H_p$-property, later  called  R-property by J.-P.\ Thouvenot \cite{JT}. This property, to which we will come back with more detail in the sequel, has important dynamical consequences, mainly expressed by a restriction on the possible joining measures of a system having the R-properties with other systems, and in particular with itself.

 A {\em joining} between two dynamical systems $(T,X,\mathscr{B},\mu)$ and  $(S,Y,\mathscr{C},\nu)$,  $(X,\mathscr{B},\mu)$ and   $(Y,\mathscr{C},\nu)$ being standard Borel probability spaces,  is a measure $\rho$ on $X \times Y$ invariant by $T \times S$ whose marginals on $X$ and $Y$ are $\mu$ and $\nu$. The definition for flows is similar. An important notion in Ratner's theory is that of  finite extension joinings (FEJ). 
\begin{definition}\label{fin}{\em An ergodic flow $(T_t)_{t\in\R}$  is said to have {\em FEJ-property}, acronym for {\em finite extension joining}, if for every ergodic flow $(S_t)_{t\in\R}$ acting on $(Y,\mathscr{C},\nu)$ and every ergodic joining $\rho$ of $(T_t)_{t\in\R}$ and $(S_t)_{t\in\R}$ different from the product measure $\mu \times \nu$, $\rho$ yields a flow which is a   finite extension of $(S_t)_{t\in\R}$.}
\end{definition}
It was shown  in  \cite{RT} that a mixing flow with  FEJ-property is mixing of all orders. Moreover, it was proved in \cite{Rat} that a flow with R-property has the FEJ-property. It follows that mixing flows with the R-property are mixing of all orders. Since the R-property for horocycle flows stemmed from polynomial shear  along the orbits, and since Kochergin flows displayed a similar polynomial shear along the orbits, the idea that special flows over rotations may enjoy the R-property, and thus be multiple mixing, was then suggested by J-P.\ Thouvenot in the 1990's (see p.\ 2 in \cite{Fr-Lem}).

However, whether natural classes of special flows (not necessarily mixing) over irrational rotations may have the R-property remained open until K.\ Fr\c{a}czek and M.\ Lema\'nczyk \cite{Fr-Lem,Fra-Lem} showed that a generalized R-property holds in some classes of special flows with roof functions of bounded variation (which, by \cite{Koc1}, are not mixing). More precisely, they have introduced a weaker notion than the R-property, called {\it weak Ratner} or WR-property  that however still implies the FEJ-property (see Definition \ref{wrs} and the comment after it) .

Unfortunately, in the mixing examples of special flows under piecewise convex functions with singularities such as Arnol'd and Kochergin flows,  the shear may occur very abruptly as orbits approach the singularity and this may prevent them from having the weak Rather property. Indeed, we believe that these flows do not have the WR-property.  That this should be true is corroborated  by the following result that shows that Kochergin flows, in the context  of bounded type frequency in the base  (that is {\it a priori} favorable  to controlling the shear), do not have the WR-property. 

\begin{theorem}\label{aby} Let $\alpha\in\T$ be irrational with bounded partial quotients and $f(x)=x^\gamma+r$, $-1<\gamma<0$, $r>0$. Then the special flow  above the circle rotation  $R_\a$  and under the ceiling function $f$ does not have the WR-property.
\end{theorem}
We denote by $\T$ the circle $\R / \Z$. We refer to Section \ref{sec.swr.special} for the exact definition of special flows.
Theorem \ref{aby} has another consequence. It is known that every horocycle flow $(h_t)_{t\in\R}$ is loosely Bernoulli \cite{Rat3}; therefore, for every irrational $\alpha$, there exists a positive function in $f\in L^1(\T)$ such that $(h_t)_{t\in\R}$ is measurably isomorphic to $(T_t^f)_{t\in\R}$ \cite{Or-Ru-We}. It follows from \cite{Koc1} and the fact that $(h_t)_{t\in\R}$ is mixing that $f$ is of unbounded variation. Moreover, by~\cite{OS},  $f$ can be made $C^1$ except for one point. Since the R-property implies the WR-property and  the R-property is an isomorphism invariant, no special flow as in Theorem \ref{aby} is isomorphic to a horocycle flow. Actually, this line of thought can be extended to show that horocycle flows are never isomorphic to special flows above an irrational rotation and under a roof function that is convex and $C^2$ except at one point. For the latter result, one needs to introduce the concept of {\it strong Ratner property}, which is also an isomorphism invariant, that specifies the occurrence of slow divergence of nearby orbits to the first time when the orbits do split apart. This will be dealt with in a future work.  


To bypass Theorem~\ref{aby} and still use controlled  divergence of orbits to show multiple mixing, our approach will be to further weaken the WR-property. Namely, we introduce the SWR-property, which  stands for {\em switchable weak Ratner} property, that assumes that a pair of nearby points displays the WR-Property either under forward iteration in time or under backward iteration, and this depending on the pair of points. We show that the SWR-property is sufficient to guarantee the same FEJ consequences as the Ratner or the weak Ratner property.  Consequently, a mixing flow enjoying the SWR-property is mixing of all orders.

The main idea in showing that Arnold and Kocergin special flows may have the SWR-Property is the following. The main contribution to the shear between orbits is due to the visits of the flow lines to the neighborhood of the singularities.  
With the representation of these flows as special flows above irrational rotations, the shear is translated into the divergence between the Birkhoff sums of the roof functions for nearby points, and this divergence is mainly due to the visits under the base rotation to the neighborhoods of the points where the roof function has its singularities. If the base rotation angle $\a$ is of bounded type two nearby points will stay sufficiently far from the singularity either when they are iterated forward or when they are iterated backward.   
In the case of ceiling functions with only logarithmic singularities we are also able to exploit the progressive contribution to the shear of these visits to the singularities to obtain multiple mixing for a 
 full measure set of numbers $\a$. 

\bigskip

 We will now describe the ceiling functions that will be considered in the sequel and state our exact results on the SWR-property and multiple mixing.

\begin{definition}{\em Let $h$ be a positive function $h\in C^2(\T\setminus\{0\})$, decreasing on $(0,1)$, $\lim_{x\to 0^+}h(x)=+\infty$, $h'$ increasing on $(0,1)$. 
 Let $f\in C^2(\T\setminus\{a_1,....,a_k\})$ for some numbers $a_1,...,a_k\in \T$.  We say that $f$ has {\it singularities of type $h$} at $\{a_1,...,a_k\}$ if  
\begin{equation}\label{asu}\lim_{x\to a_i^+}\frac{f'(x)}{h'(x-a_i)}=A_i\;\; \text{and}\;\; \lim_{x\to a_i^-}\frac{f'(x)}{h'(a_i-x)}=-B_i,\end{equation}
for some numbers $A_i,B_i\geq 0$, $i=1,...,k$. }
\end{definition}

Notice that in this definition $h$ may only reflect a domination on the singularities of $f$ since the coefficients $A_i,B_i$ may be equal to zero at some or at all $i$'s.
 
In all the sequel, we will consider $\a \in \R \setminus \Q$ and let $(q_s)$ be the sequence of denominators of the best rational approximations of $\a$. Namely $(q_s)$ is the unique increasing sequence such that $q_0=1$ and $\|q_s \a\| <\|k\a\|$ for any $k <q_{s+1}$, $k\neq q_s$. We recall that 
\begin{equation} \label{red0} \frac{1}{q_s+q_{s+1}} \leq  \|q_s \a\| \leq \frac{1}{q_{s+1}} \end{equation}

Our results can deal with functions having several singularities but require a non resonance condition of these singularities with $\a$.

\begin{definition}\label{gen.pos}{\em[Badly approximable singularities]  Given $\a \in \R \setminus \Q$, we will say that  $\{a_1,....,a_k\}$ are {\it badly approximable by $\a$} if there exists $C>1$ such that for every $x\in \T$ and every $s\in \N$, 
there exists at most one $i_0\in \{0,...,q_s-1\}$ such that
\begin{equation}\label{basic}x+i_0\alpha\in \bigcup_{i=1}^k\left[\frac{-1}{2Cq_s}+a_i,a_i+\frac{1}{2Cq_s}\right].\end{equation}}
\end{definition} 

\begin{remark}{\em It was shown in \cite[Lemma 3]{Fr-Lem-Les} that if $a_i-a_j\in (\Q+\Q\alpha)\setminus(\Z+\Z\alpha)$ whenever $i\neq j$ then $\{a_1,....,a_k\}$  are badly approximable by $\a$.
}\end{remark}

Note that if there is only one singularity, that is $k=1$, then by \eqref{red0} it is always badly approximable by $\a$.  The following shows that for $k \geq 2$ the set of singularities that are  badly approximable by $\a$  is a thick set in $[0,1]^k$.

\begin{lemma}  Let $\a \in \R \setminus \Q$. For any $k \in \N$, the set  $E \subset [0,1]^k$ of $k-tuples$ $(a_1,\ldots,a_k)$ that are badly approximable by $\a$ is a product of sets of full Hausdorff dimension in $[0,1]$. \end{lemma}

\begin{proof} Define 
$$B(\a) := \{Ęb \in \R : \exists C>0, \forall k \in \Z \setminus 0 :  \| k \a - b\| \geq  \frac{1}{C |k|}\}$$
Then if  $(a_1,\ldots,a_k)$  are such that $a_i-a_j \in B(\a)$ for any $i,j \in \{1,\ldots,k\}$, $i\neq j$, then  $(a_1,\ldots,a_k)$ are
badly approximable by $\a$.

But it was proven in \cite{tseng} (see also \cite{einsiedler-tseng}) that the set $B(\a)$ is a winning set in the sense of Schmidt (see \cite{tseng,einsiedler-tseng} and references therein). A winning set is of full Huasdorff dimension. Moreover, for a winning set $B \in \R$ we have that for any $x_1,\ldots,x_n$ the set $\cap_{s=1}^n (x_s+B)$ is winning. So, if $a_1, \ldots,a_l$ are such that $a_i-a_j \in B(\a)$ for any $i,j \in \{1,\ldots,l\}$, $i\neq j$, then the set of $a \in [0,1]$ such that $a \in \cap_{s=1}^l (a_s + B(\a))$ is winning which means that $a_1, \ldots,a_l,a_{l+1}$ are  badly approximable  by $\a$ for a winning set of $a_{l+1}$. The statement of the Lemma follows then by induction and because a single $a_1$ is always badly approximable by $\a$. \end{proof}

Our results deal with two types of singularities. Theorem \ref{main} deals with logarithmic like singularities, while Theorem \ref{boun} deals with the case of at least one dominant power like singularities.

\subsection{Logarithmic like singularities} 

In the case of logarithmic like singularities, the following theorem holds.

\begin{theorem}\label{main}  Let $\a \in \R \setminus \Q$ and $f\in C^2(\T\setminus\{a_1,....,a_k\})$ with the singularities $\{a_1,\ldots,a_k\}$ {\it of type $h$} and badly approximable  by $\a$, with some associated constant $C>1$.  Assume that  $\sum_{i=1}^k A_i\neq \sum_{i=1}^k B_i$ and that there exist a constant $m_0>0$ and a sequence $(x_s)$ such that 
   for every $s\in\N$, we have $x_s < \frac{1}{q_s}$ and 
\begin{enumerate}

    \item  $\lim_{s\to +\infty} \frac{h'(\frac{x_s}{4C})}{q_sh(\frac{1}{2q_s})}=0, \quad \lim_{s\to +\infty}x_sq_sh(\frac{1}{2q_s})=+\infty$;
	 \item $\sum_{i\notin K_\alpha}q_ix_i<+\infty$, where $K_\alpha:=\{s\in \N\;:\; q_{s+1}< \frac{1}{x_s}\}$;
	 \item  $h(\frac{1}{2q_s})/h(\frac{1}{2q_{s+1}})>m_0$.
\end{enumerate}
Then  $(T_t^f)_{t\in \R}$ has the SWR-property.
\end{theorem}

To describe a consequence of Theorem \ref{main} (see Corollary \ref{ln}), set $h(x)=-\ln(x),$ for $x\in (0,1]$.  

For $\a \in \RQ$, let $K_{\alpha}:=\{n\in \N\;:\;q_{n+1}<q_n\log^{\frac{7}{8}}(q_n)\}$.  
We then define in view of  $1.$ and  $2.$ of Theorem \ref{main} 
\begin{equation*}\label{mate}\mathcal{E}:=\left\{\alpha\in \T \setminus \Q \;:\; \sum_{i\notin K_\alpha}\frac{1}{\log^{\frac{7}{8}}q_i}<+\infty\right\}.\end{equation*}

To have 3. of Theorem \ref{main} it suffices to assume that $\a$ is Diophantine : for $\tau \geq 0$ define the set of Diophantine numbers $\a$ of exponent $\tau$ to be 
$$ DC(\tau) := \{\alpha\in \T \setminus \Q \::\: \ \exists_{r_{\alpha}\in \N} \ \;\forall_{n\in \N}, \ q_{n+1}<r_\alpha q^{1+\tau}_{n}\}.$$
The set of Diophantine numbers is then
$$ \mathcal D := \{\alpha\in \T \setminus \Q \::\: \exists \tau \geq 0, \  \a \in DC(\tau)\}.$$
An equivalent definition of $DC(\tau)$ is that for any $p,q\in \Z \times \N^*$ we have that $|\a-\frac{p}{q}|\geq \frac{C(\a)}{q^{2+\tau}}$ for some $C(\a)>0$. 

\begin{corollary}\label{ln} Consider $h(x)=-\log(x)$. Let $\a \in \cD \cap \cE$. Let $f\in C^2(\T\setminus\{a_1,....,a_k\})$ with the singularities $\{a_1,\ldots,a_k\}$ {\it of type $h$} and badly approximable  by $\a$, with some associated constant $C>1$.  Assume that  $\sum_{i=1}^k A_i\neq \sum_{i=1}^k B_i$. Then $(T_t^f)_{t\in\R}$ has the SWR-Property and is mixing of all orders. 

\end{corollary}

\begin{proof} We take for $x_s$ the sequence $\frac{1}{q_s \log^{\frac{7}{8}}q_s}$ and easily check the hypothesis of Theorem \ref{main}. Therefore  $(T_t^f)_{t\in \R}$ has the SWR-property. On the other hand, it was shown in \cite{Koc4} that $(T_t^f)_{t\in \R}$ is mixing. Multiple mixing then follows from Theorem \ref{cons} and the FEJ-property. \end{proof}

Corollary \ref{ln} covers a set of full Lebesgue measure of rotation angles $\a$. Indeed, it is known that the set of Diophantine numbers $\mathcal D$ has full Lebesgue measure, and we will  prove in Appendix \ref{EE} the following result. Denote by $\lambda$ the Haar measure on $\T$.
\begin{proposition} \label{propEE} It holds that $\lambda(\mathcal{E})=1$.
\end{proposition}

\subsection{Power like singularities}

Now, we will deal with power like singularities. We suppose   $f\in C^2(\T\setminus\{a_1,....,a_k\})$ with singularities $\{a_1,....,a_k\}$ of type $h$. We divide the set $\{a_1,...,a_k\}$ of singularities into two subsets : $F$ the set of weak singularities, and $E$ the set of  strong singularities of type not less than $h$ : namely, we suppose  $F=\{a_1,...,a_v\}$ and $E=\{a_{v+1},...,a_k\} \in E$ are such that,   $A_i^2+B_i^2>0$ in \eqref{asu} for $i \in \{v+1,\ldots, k \}$, while each $a_i\in F$ is a singularity for $f$ of type $g_i$ with $g_i$ a positive function in $ C^2(\T\setminus\{0\})$, decreasing on $(0,1)$ with $g_i'$ increasing and such that
\begin{equation}\label{gi}\lim_{x\to a^+_i}\left|\frac{f'(x)}{g_i'(x-a_i)}\right|\;\;\;\text{and}\;\;\;\lim_{x\to a^-_i}\left|\frac{f'(x)}{g'(a_i-x)}\right|\;\;\text{exist and are finite},\end{equation}
$\lim_{x\to 0^+}\frac{g_i(x)}{h(x)}=0$ and for every $s\in \N$ there exists $x_{i,s}\in \T$, $x_{i,s}>\frac{1}{q_sh(\frac{1}{2q_s})}$, such that $\lim_{s\to +\infty}\frac{g_i'(x_{i,s})}{h'(\frac{1}{2q_s})}=0$ and $\sum_{i=1}^{+\infty}q_sx_{i,s}<+\infty$.

We always assume that $E$ is not empty.


\begin{theorem}\label{boun} Let $\alpha$ be irrational with bounded partial quotients, that is, $\a \in DC(0)$. Assume that $\{a_{v+1},...,a_k\}$ are badly approximable  by $\a$ with some constant $C >1$. 
 Assume that there exist constants $D_1,D_2>0$ such that
for every $s\in\N$
\begin{equation}\label{dd}
D_2>\frac{-h'(\frac{1}{C^4q_s})}{q_sh(\frac{1}{2q_s})}>D_1\;\;\;\text{and}\;\;\; \frac{h(\frac{1}{2q_s})}{h(\frac{1}{2q_{s+1}})}>D_1.
\end{equation}
Then $(T_t^f)_{t\in \R}$ has the SWR-property.
\end{theorem}

\begin{corollary}\label{pow} Let $\a \in DC(0)$. Let $f\in C^2(\T\setminus\{a_1,....,a_k\})$ with all the singularities $\{a_1,....,a_k\}$ of power-like type $x^{\gamma_i}$ from the left and   $x^{\delta_i}$ from the right, $-1<\gamma_i,\delta_i<0$. Let $\gamma=\min_{1\leq i\leq k}\{\gamma_i,\delta_i\}$, $E=\{a_i\;:\; \min \{\gamma_i,\delta_i\}=\gamma\}$. Then, if the points in $E$ are badly approximable by $\a$, we have that  $(T_t^f)_{t\in \R}$ has the SWR-Property and is mixing of all orders. 
 \end{corollary}
 Note that there are no combinatorial assumptions on the weak singularities $a_j\notin E$.

\begin{proof} [Proof of Corollary \ref{pow}] Take $x_s=\frac{1}{s^2 q_s}$ and easily check the hypothesis of Theorem \ref{boun}. This gives the SWR-Property . Mixing of $(T_t^f)_{t\in \R}$ was established in \cite{Koc2}. Multiple mixing then follows from Theorem \ref{cons} and the FEJ-property. \end{proof}

\subsection{Plan of the paper} In Section \ref{sec.swr} we introduce the SWR-Property and study its joinings consequences. In Section \ref{sec.swr.special} we give a criterion involving the Birkhoff sums of the ceiling function that guarantees that a special flow above an isometry has the SWR-property. The treatment of these sections is similar to \cite{Fr-Lem,Fra-Lem}. In Section \ref{sec.proofs} we study the Birkhoff sums of logarithmic like and power like functions and prove Theorems \ref{main} and \ref{boun}. Section \ref{abs} is devoted to the proof of Theorem \ref{aby} on the absence of the SWR-Property for a subcalss of Kochergin flows. Finally Appendix \ref{App.A} is devoted to the proof that the set of frequencies for which Theorem \ref{main} holds has full Lebesgue measure.

\bigskip

\noindent{\bf Acknowledgments.}

\indent The second author would like to thank Professor Mariusz Lema\'nczyk for all his patience, help and deep insight.
The authors would also like to thank Krzysztof Fr\c{a}czek, Mariusz Lema\'nczyk and Jean-Paul Thouvenot for valuable discussions on the subject.\\

\indent The results of Section \ref{sec.proofs} have been obtained by the two authors independently and the results of Section \ref{abs} by the second. The two authors decided to include Section \ref{abs} in this work because it is an integral part of the problems concerning Ratner's property for this class of special flows.

\section{The SWR- property} \label{sec.swr}
Let $(X,\mathscr{B},\mu)$ be a probability standard Borel space. We additionally assume that $X$ is a complete metric space with a metric $d$. Let $(T_t)_{t\in\R}$ be an ergodic flow acting on $(X,\mathscr{B},\mu)$.
\begin{definition}[cf.\ \cite{Fra-Lem}, Definition 4] \label{wrs} {\em Fix $t_0\in \R_+$ and a compact set $P\subset \R\setminus \{0\}$. One says that the flow has the {\em switchable} $R(t_0,P)$-property if for every $\epsilon>0$ and $N\in \N$ there exist $\kappa=\kappa(\epsilon)$, $\delta=\delta(\epsilon,N)$ and a set $Z=Z(\epsilon,N)\subset \mathscr{B}$ with $\mu(Z)>1-\epsilon$ such that for any $x,y\in Z$ with $d(x,y)<\delta$, $x$ not in the orbit of $y$ there exist $M=M(x,y),L=L(x,y)\in \N$ with $M,L>N$ and $\frac{L}{M}\geq \kappa$ and $p=p(x,y)\in P$ such that

\begin{equation}\label{fir}
\frac{1}{L}\big|\{n\in[M,M+L]: d(T_{nt_0}(x),T_{nt_0+p}(y))<\epsilon\}\big|>1-\epsilon\end{equation}
\begin{center}
    or
\end{center}
\begin{equation}\label{sec}\frac{1}{L}\left|\{n\in[M,M+L]: d(T_{n(-t_0)}(x),T_{n(-t_0)+p}(y))<\epsilon\}\right|>1-\epsilon.\end{equation}

If the set of $t_0>0$ such that the flow $(T_t)_{t\in\R}$ has the switchable $R(t_0,P)$-property is uncountable, the flow is said to have {\em SWR}-property.}
 \end{definition}

For the sake of completeness, compare the SWR-property with the definition of the WR-property \cite{Fra-Lem}. To have WR-property, we fix $P\subset \R\setminus \{0\}$ and $t_0\in \R$. $(T_t^f)_{t\in\R}$ has $R(t_0,P)$ property if in Definition \ref{wrs}, (\ref{fir}) holds (the condition (\ref{sec}) is not taken into account) and $(T_t^f)_{t\in\R}$ has WR-property if the set of $t_0\in \R$ such that $(T_t^f)_{t\in\R}$ has $R(t_0,P)$ property is uncountable. Consequently, SWR-property is weaker than WR-property (and as Theorem~\ref{aby} shows, it is strictly weaker).

Now, again for sake of completeness, we will present a detailed proof (using some facts proved in \cite{Fra-Lem}) of the fact under the ``continuty'' assumption on orbits (see below) that SWR-property has FE-property as  the original $H_p$-property introduced by M.\ Ratner \cite{Rat}.

 We will state a lemma which is a simple consequence of Lemma 5.2. in \cite{Fra-Lem}.

\begin{lemma}\label{bet} Let $T,S:(X,\mathscr{B},\mu)\to (X,\mathscr{B},\mu)$ be two ergodic automorphisms and let $A\in\mathscr{B}$. For any $\epsilon,\delta,\kappa>0$ there exist $N=N(\epsilon,\delta,\kappa)$ and a measurable set $Z=Z(\epsilon,\delta,\kappa)$ with $\mu(Z)>1-\delta$ such that for any $M,L\geq N$ with $\frac{L}{M}\geq \kappa$ and any $x\in \Z$ we have
$$\left|\frac{1}{L}\sum_{i=M}^{M+L}\chi_A(T^ix)-\mu(A)\right|<\epsilon$$
and
$$\left|\frac{1}{L}\sum_{i=M}^{M+L}\chi_A(S^ix)-\mu(A)\right|<\epsilon.$$
\end{lemma}

We will add one more natural condition on the flow $(T_t)_{t\in\R}$ which can be viewed as ``continuity'' on orbits. The flow $(T_t)_{t\in\R}$ is called {\em almost continuous} \cite{Fra-Lem} if for every $\epsilon>0$ there exists a set $X:=X(\epsilon)$ with $\mu(X)>1-\epsilon$ such that for every $\epsilon'>0$ there exists $\delta'>0$ such that for every $x\in X$, we have $d(T_t(x),T_{t'}(x))<\epsilon'$ for $t,t'\in [-\delta,\delta]$.

For the definition and  properties of joinings, we refer the reader to \cite{JT} or \cite{Gl}. Our goal is now to prove the following result.
\begin{theorem}\label{cons} Let $(T_t)_{t\in\R}$ be a weakly mixing flow acting on a probability standard Borel space $(X,\mathscr{B},\mu)$. Assume that $(T_t)_{t\in\R}$ satisfies the SWR-property. Let $(S_t)_{t\in\R}$ be an ergodic flow acting on a probability standard Borel space $(Y,\mathscr{C},\nu)$ and let $\rho\in J((T_t)_{t\in\R},(S_t)_{t\in\R})$ be an ergodic joining. Then either $\rho$ is equal to $\mu\otimes\nu$ or is a finite extension of the measure $\nu$.
\end{theorem}

\noindent To prove this theorem we need some lemmas from \cite{Fra-Lem}.

\begin{lemma}\label{comp} Let $(T_t)_{t\in \R}$ be an ergodic almost continuous flow acting on $(X,\mathscr{B},\mu)$, and $(S_t)_{t\in\R}$ be another ergodic flow acting on $(Y,\mathscr{C},\nu)$. Let $\rho\in J((T_t)_{t\in\R},(S_t)_{t\in\R})$ be such that $\rho$ is ergodic for automorphisms $T_1\times S_1$ (hence, for $T_{-1}\times S_{-1}$). Let $P\subset \R$ be non-empty and compact. Let $A\in \mathscr{B}$ be such that $\mu(\partial A)=0$ and $B\in\mathscr{C}$. Then, for every $\epsilon,\delta,\kappa>0$ there exist a natural number $N=N(\epsilon,\delta,\kappa)$ and a set $Z=Z(\epsilon,\delta,\kappa)\subset \mathscr{B}\otimes\mathscr{C}$ with $\rho(Z)>1-\delta$ such that for any $\N\ni M,L\geq N$ with $\frac{L}{M}\geq \kappa$ and any $p\in P$, we have
$$\left|\frac{1}{L}\sum_{j=M}^{M+L}\chi_{T_{-p}A\times B}(T_jx,S_jy)-\rho(T_{-p}A\times B)\right|<\epsilon$$
and
$$\left|\frac{1}{L}\sum_{j=M}^{M+L}\chi_{T_{-p}A\times B}(T_{-j}x,S_{-j}y)-\rho(T_{-p}A\times B)\right|<\epsilon$$
for every $(x,y)\in Z$.
\end{lemma}
\begin{proof} The proof is a simple consequence of Lemma 5.4. in \cite{Fra-Lem}. One uses this lemma first for the flows $(T_t)_{t\in\R}$ and $(S_t)_{t\in\R}$ and ergodic joining $\rho\in J((T_t)_{t\in\R},(S_t)_{t\in\R})$ to get, for $\epsilon,\frac{\delta}{2}, \kappa>0$, a natural number $N_+\in \N$ and a set $Z_+\subset \mathscr{B}\otimes\mathscr{C}$ with $\rho(Z_+)>1-\frac{\delta}{2}$. Then, for flows $(T_{-t})_{t\in\R}$ and $(S_{-t})_{t\in\R}$ and the same ergodic joining $\rho$ to get, for $\epsilon,\frac{\delta}{2},\kappa>0$, a natural number $N_-\in \N$ and a set $Z_-\subset \mathscr{B}\otimes\mathscr{C}$ with $\rho(Z_-)>1-\frac{\delta}{2}$. To finish the proof one takes $N:=\max(N_+,N_-)$
and $Z=Z_+\cap Z_-$. \end{proof}
\vspace{1ex}

Next lemma is used in the proof of Theorem~3 in \cite{Rat}.
\begin{lemma}\label{rto} Let $(T_t)_{t\in\R}$ and $(S_t)_{t\in\R}$ be two ergodic flows. Let $\rho\in J^e((T_t)_{t\in\R},(S_t)_{t\in \R})$ be an ergodic joining. Then if there exists a set $V$ with $\rho(V)>0$ such that for any points $(x,y),(x',y)\in V$ either $x$ is in the orbit of $x'$ or $d(x,x')>c_0$ for some constant $c_0>0$, then $\rho$ is a finite extension of $\nu$.
\end{lemma}
In what follows, we consider only $(X,d)$ be a $\sigma$-compact metric space. Let $A\in \mathscr{B}$. For $\eta>0$ we denote by $V_\eta(A):=\{x\in X:\; d(x,A)<\eta\}$.
\begin{lemma}\label{nei}[cf.\ \cite{Fra-Lem}]  For any $A\in \mathscr{B}$ there exists $R\subset (0,+\infty)$ such that $(0,+\infty)\setminus R$ is countable and $\mu(\partial V_{\eta}(A))=0$ for $\eta\in R$. It particular,  there exists a dense family $(B_i)_{i\geq 1}$ in $\mathscr{B}$ with the property $\mu(\partial B_i)=0$ for every $i\in \N$.
\end{lemma}
\begin{proof}[Proof of Theorem \ref{cons}.] Let $\rho\in J((T_t)_{t\in\R},(S_t)_{t\in \R})$ be an ergodic joining and $\rho\neq \mu\times \nu$. Assume that $(T_t)_{t\in\R}$ has the switchable $R(t_0,P)$-property and $\rho$ is ergodic for $T_{t_0}\times S_{t_0}$ (then $\rho$ is ergodic for $T_{-t_0}\times S_{-t_0}$). Such $t_0>0$ always exists because an ergodic flow can have at most countably many non-ergodic time automorphisms. For simplicity of notation, we assume $t_0=1$. Let $\{B_i\}_{i\geq 1}$ and $\{C_i\}_{i\geq 1}$ be two countable dense families in the $\sigma$-algebras $\mathscr{B}$ and $\mathscr{C}$, respectively. Consider the following real function:
$$\R\ni t\to k(t):=\sum_{i,j\geq 1}(1/2^{i+j})|\rho(T_t(B_i)\times C_j)-\rho(B_i\times C_j)|.$$
As in Lemma 5.4. in \cite{Fra-Lem}, we conclude that $k:\R\to\R$ is a continuous function and for any $t\in\R$, $k(t)>0$. Indeed, it follows by the fact that if for some $r\in\R\setminus\{0\}$ we have for any $i,j\in \N$ $\rho(T_{r}(B_i)\times C_j)=\rho(B_i\times C_j)$ then $\rho$ is product measure (recall that $T_t$    is assumed to be weak mixing hence every time $r$ of the flow is ergodic). \\
The set $P\subset \R\setminus \{0\}$ is compact, therefore there exists $\epsilon>0$ such that $k(p)>\epsilon$ for any $p\in P$. It follows by the definition of the function $k$ that there exists a number $R:=R(\epsilon)$ such that
$$\sum_{i,j\geq 1}^R(1/2^{i+j})|\rho(T_p(B_i)\times C_j)-\rho(B_i\times C_j)|>\epsilon/2$$
for every $p\in P$. Therefore, for every $p\in P$, there exist $1\leq i,j\leq R$ such that $|\rho(T_p(B_i)\times C_j)-\rho(B_i\times C_j)|>\epsilon$.\\
By Lemma \ref{nei}, there exists $\epsilon'<\frac{\epsilon}{8}$ such that for every $1\leq i\leq R$
$$\mu(V_{\epsilon'}(B_i)\setminus B_i)<\epsilon,\; \text{and}\; \mu(\partial V_{\epsilon'}(B_i))=0.$$
It follows by the fact that $\rho$ is a joining that
\begin{equation}\label{prop}
|\rho(V_{\epsilon'}(B_i)\times C_j)-\rho(B_i\times C_j)|<\frac{\epsilon}{2}\;\text{and}\; |\rho(S_{-t}V_{\epsilon'}(B_i)\times C_j)-\rho(S_{-t}B_i\times C_j)|<\frac{\epsilon}{2},
\end{equation}
for $1\leq i,j\leq R$ and every $t\in\R$.
By the switchable $R(1,P)$-property, let $\kappa:=\kappa(\epsilon')$.
By Lemma~\ref{bet} applied to $\frac{\epsilon}{8},\frac{1}{8},\kappa$, the sets $V_{\epsilon'}(B_i)\times C_j$, $1\leq i,j\leq R$, and to automorphisms $T_1\times S_1$ and $T_{-1}\times S_{-1}$, we get $N_1\in \N$ and a set $U_1\in \mathscr{B}\otimes\mathscr{C}$ with $\rho(U_1)>\frac{7}{8}$, such that for every $L,M\geq N_1$ with $\frac{L}{M}\geq \kappa$ and every $(x,y)\in U_1$, we have
\begin{equation}\label{1}
\left|\frac{1}{L}\sum_{k=M}^{M+L}\chi_{V_{\epsilon'}(B_i)\times C_j}(T^kx,S^ky)-\rho(V_{\epsilon'}(B_i)\times C_j)\right|<\frac{\epsilon}{8}
\end{equation}
\begin{equation}\label{2}
\left|\frac{1}{L}\sum_{k=M}^{M+L}\chi_{V_{\epsilon'}(B_i)\times C_j}(T^{-k}x,S^{-k}y)-\rho(V_{\epsilon'}(B_i)\times C_j)\right|<\frac{\epsilon}{8}.
\end{equation}
Next, by Lemma \ref{comp} applied to $\frac{\epsilon}{8},\frac{1}{8},\kappa>0$ and the sets $B_i\times C_j$, $1\leq i,j\leq R$, there exist $N_2\in \N$ and a set $U_2\subset \mathscr{B}\otimes \mathscr{C}$ with $\rho(U_2)>\frac{7}{8}$ such that for every $L,M\geq N_2$ with $\frac{L}{M}\geq \kappa$ and any $p\in P$, we have
\begin{equation}\label{3}\left|\frac{1}{L}\sum_{k=M}^{M+L}\chi_{T_{-p}B_i
\times C_j}(T_kx,S_ky)-\rho(T_{-p}B_i\times C_j)\right|<\frac{\epsilon}{8}
\end{equation}
and

\begin{equation}\label{4}\left|\frac{1}{L}\sum_{k=M}^{M+L}\chi_{T_{-p}B_i\times C_j}(T_{-k}x,S_{-k}y)-\rho(T_{-p}B_i\times C_j)\right|<\frac{\epsilon}{8}.
\end{equation}
It follows that if we set $N_0:=\max(N_1,N_2)$ and $U_0:=U_1\cap U_2$, then $\rho(U_0)>\frac{1}{2}$ and for every $L,M\geq N_0$ with $\frac{L}{M}\geq \kappa$, any $p\in P$, the equations (\ref{1}), (\ref{2}), (\ref{3}), (\ref{4}) are satisfied for every $(x,y)\in U_0$. Using the switchable $R(1,P)$-property with $\epsilon'>0$ and $N_0\in \N$, we obtain $\delta=\delta(\epsilon',N_0)$ and $Z=Z(\epsilon',N_0)$ with $\mu(Z)>1-\epsilon'$. Now, we will use Lemma \ref{rto} with the set $U:=U_0\cap(Z\times Y)$ (then of course $\rho(U)>\frac{1}{4}$) and $\delta_0=\delta(\epsilon', N_0)$ to prove that for every $(x,y),(x',y)\in U$, $d(x,x')\geq\delta_0$. Assume on the contrary that $d(x,x')<\delta_0$. Then by the switchable  R(1,P)-property, there exist $L_0,M_0>N_0$ with $\frac{L_0}{M_0}\geq \kappa$ and $p\in P$ such that
$$
\frac{1}{L_0}\big|\{n\in[M_0,M_0+L_0]: d(T_{n}(x),T_{n+p}(x'))<\epsilon'\}\big|>1-\epsilon'$$
\begin{center}
or
\end{center}
$$\frac{1}{L_0}\big|\{n\in[M_0,M_0+L_0]: d(T_{-n}(x),T_{-n+p}(x'))<\epsilon'\}\big|>1-\epsilon'.$$
Assume that the first inequality is satisfied. We will use equations (\ref{1}) and (\ref{3}) (in case the second one is satisfied, we use equations (\ref{2}) and (\ref{4})). Let $1\leq i_p, j_p\leq R$ be the numbers which satisfy $|\rho(T_p(B_{i_p})\times C_{j_p})-\rho(B_{i_p}\times C_{j_p})|>\epsilon$.  Let $K=K(x,x',p):= \{n\in[M_0,M_0+L_0]: d(T_{n}(x),T_{n+p}(x'))<\epsilon'\}$. It follows that if $k\in K$ and $T_{k+p}x'\in A_i$ then $T_kx\in V_{\epsilon'}(A_i)$. Therefore
\begin{multline}
\rho(T_{-p}B_{i_p}\times C_{j_p})\leq \frac{1}{L_0}\sum_{k=M_0}^{M_0+L_0}\chi_{T_{-p}B_{i_p}\times C_{j_p}}(T^kx',S^ky)+\frac{\epsilon}{8}\leq\\
\frac{\epsilon' L_0}{L_0}+\frac{1}{L_0}\sum_{k=M_0}^{M_0+L_0}\chi_{V_{\epsilon'}(B_{i_p})\times C_{j_p}}(T^kx,S^ky)+\frac{\epsilon}{8}\leq \\
\frac{\epsilon}{2}+\rho(V_{\epsilon'}(B_{i_p})\times C_{j_p})<\epsilon + \rho(B_{i_p}\times C_{j_p}).
\end{multline}
A similar arguments show that $\rho(B_{i_p}\times C_{j_p})<\epsilon+\rho(T_{-p}B_{i_p}\times C_{j_p})$ and consequently, $|\rho(B_{i_p}\times C_{j_p})-\rho(T_{-p}B_{i_p}\times C_{j_p})|<\epsilon$. This contradicts our assumption that $|\rho(T_p(B_{i_p})\times C_{j_p})-\rho(B_{i_p}\times C_{j_p})|>\epsilon$ is satisfied. Therefore, for any $(x,y),(x',y)\in U$ we have  $d(x,x')\geq\delta_0$ and an application of Lemma~\ref{rto} completes the proof. \end{proof}
\section{SWR-property for special flows} \label{sec.swr.special}
In this section, we will prove a sufficient condition for SWR-property in the case of special flows over an ergodic isometry We start by recalling the definition of special flows. 
Let $T$ be an automorphism $(X,\mathscr{B},\mu)$.  Let $f\in L^1(X,\mu)$ such that $f>0$.  The {\emph special flow} $(T_t^f)_{t\in\R}$ defined above $T$   and under the ceiling function $f$ is given by
\begin{eqnarray*}
X \times \R / \sim  &  \rightarrow &  X \times \R / \sim  \\
 (x,s) & \rightarrow & (x,s+t), \end{eqnarray*}
where $\sim$ is the
identification 
\begin{equation}
\label{FlowSpace}
(x, s + f(x)) \sim (T(x),s) 
\end{equation}
Equivalently the flow $(T_t^f)_{t\in\R}$ is defined for  $t+s \geq 0$ (with a similar definition for negative times) by 
$$T_t^f(x,s) = (T^n x, t+s-f^{(n)}(x))$$ 
   where $n$ is the unique integer such that 
\begin{equation}
\label{D-C}
   f^{(n)}(x) \leq t+s < f^{(n+1)}(x)
\end{equation}    
and
$$
f^{(n)}(x)=\left\{\begin{array}{ccc}
f(x)+\ldots+f(T^{n-1}x) &\mbox{if} & n>0\\
0&\mbox{if}& n=0\\
-(f(T^nx)+\ldots+f(T^{-1}x))&\mbox{if} &n<0.\end{array}\right.$$

If $T$ preserves a unique probability measure $\mu$ then the special flow will preserve a unique probability measure that is the normalized product measure of $\mu$ on the base and the Lebesgue measure on the fibers. If $X$ is a matric space with a metric $d$, so is $X^f$ with the metric  $d^f((x,s),(x',s')):=d(x,x')+|s-s'|$. Moreover, it is easy to show that if $(T_t^f)_{t\in \R}$ is a special flow acting on $X^f$, then $(T_t^f)_{t\in\R}$ is almost continuous (see Section \ref{sec.swr}) with $X(\epsilon)=\{(x,s)\in X^f: x\in X, \epsilon<s<f(x)-\epsilon\}$.

The following general lemma is a direct consequence of Birkhoff ergodic theorem.
\begin{lemma}\label{erg} Let $T$ be an ergodic automorphism $(X,\mathscr{B},\mu)$.  Let $f\in L^1(X,\mu)$, $\int_Xf\,d\mu\neq0$. For every $\epsilon,\kappa>0$ there exist $N=N(\epsilon,\kappa)$ and a set $A=A(\epsilon,\kappa)$ with $\mu(A)>1-\epsilon$ such that for every $M\geq N$
\begin{equation}\label{mne}\left|\frac{1}{M}\sum_{i=1}^Mf(T^ix)-
\int_{X}f\,d\mu\right|\leq \frac{\kappa}{3}\left|\int_X f\,d\mu\right|
\end{equation}
for every $x\in A$.
\end{lemma}

\begin{remark}\label{adam}\em Assume that additionally $f$ is positive and bounded away from zero. Fix $\epsilon,\kappa>0$ ($\kappa<|\int_X f\,d\mu|<1/2$). It follows that there are constants $r_1,r_2>0$  such that if we take $x\in A$ then  for any $M,L\geq N$ with $\frac{L}{M}\geq \kappa$, we have $r_1<\frac{f^{(M)}(x)}{M}<r_2$ and
$$r_1<\frac{(1-\frac\kappa3)\int f\,d\mu(M+L)-(1+\frac\kappa3)\int f\,d\mu\cdot M}L\leq\frac{f^{(M+L)}(x)-f^{(M)}(x)}{M}<r_2.$$
\end{remark}

\begin{proposition}\label{cocy} Let $T:(X,d)\to(X,d)$ be an ergodic isometry and $f\in L^1(X,\mathscr{B},\mu)$ a positive function bounded away from zero. Let $(T_t^f)_{t\in\R}$ be the corresponding special flow. Let $P\subset \R\setminus\{0\}$ be a compact set. Assume that for every $\epsilon>0$ and $N\in \N$ there exist $\kappa=\kappa(\epsilon)$, $\delta=\delta(\epsilon,N)$ and a set $X'=X'(\epsilon,N)$ with $\mu(X')>1-\epsilon$, such that for any $x,y\in X'$ with $0<d(x,y)<\delta$ there exist $\N\ni M=M(x,y),L=L(x,y)$ with $M,L\geq N$, $\frac{L}{M}\geq \kappa$ and $p=p(x,y)\in P$ such that 
\begin{equation}\label{posi}|f^{(n)}(x)-f^{(n)}(y)-p|<\epsilon\;\;\text{for every}\;\; n\in [M,M+L]
\end{equation}
or
\begin{equation}\label{nega}|f^{(-n)}(x)-f^{(-n)}(y)-p|<\epsilon\;\text{for every}\;\; n\in [M,M+L].
\end{equation}
If $\gamma>0$ is such that the automorphism $T^f_\gamma$ is ergodic, then $(T^f_t)_{t\in \R}$ has the switchable $R(\gamma,P)$-property.
\end{proposition}
\begin{proof} Fix $\gamma>0$ such that  $T^f_\gamma$ is ergodic. Fix also $\frac{1}{\|f\|_{L^1}}>4\epsilon>0$. Apply Remark~\ref{adam} with the constants $\epsilon/4,\kappa$ to $f$ and $T$, $T^{-1}$, respectively to obtain constants $D_1,D_2>0$ such that
for $x\in A$, $\mu(A)>1-\epsilon/2$ (the set $A$ is the intersection of two relevant sets), we have
\begin{equation}\label{adam1}
D_1<\frac{f^{(M)}(x)}{M},\frac{f^{(M+L)}(x)-f^{(M)}(x)}{L},\frac{f^{(-M)}(x)}{-M},\frac{f^{(-M-L)}(x)-f^{(-M)}(x)}{-L}<D_2.
\end{equation} 
 
Fix $N>\frac{2}{D_2\epsilon^2}$. Let $\epsilon':=\min(\frac{D_1\epsilon}{8(\gamma+D_2)},\frac{\epsilon}{16})$.
Let $\kappa':=\frac{D_1}{D_2}\kappa(\epsilon')$. Let us consider the set $X(\epsilon)$ on which $(T_t^f)_{t\in \R}$ is $\frac{\epsilon}{8}$- ``almost continuous'', that is
$$X(\epsilon):=\{(x,s)\in X^f\;:\; \frac{\epsilon}{8}<s<f(x)-\frac{\epsilon}{8}\}.$$
Now, we will use ergodicity of $T^f_\gamma$ and $T^f_{-\gamma}$. It follows that there exist $N_0:=N(\epsilon)$ and a set $Z:=Z(\epsilon)$ with $\mu^f(Z)>1-\frac{\epsilon}{2}$ and for every $(x,s)\in Z$ and $n\geq N_0$
\begin{equation} \left|\frac{1}{n}\sum_{k=1}^n \chi_{X(\epsilon)}T^f_{ki}(x,s)-(1-\frac{\epsilon}{4})\right|<\frac{\kappa}{\kappa+1}\frac{\epsilon}{8}
\end{equation}
for $i=\gamma,-\gamma$. Moreover, since $f\in L^1(X,\mathscr{B},\mu)$, there exists a set $V=V(\epsilon)\subset X$ with $\mu(V)>1-\frac{\epsilon}{2}$ and such that for every $x\in V$, $f(x)<\frac{2}{\epsilon^2}$. Define the set $Z':=Z\cap\{(x,s)\in X^f\::\; x\in V\}\cap \{(x,s)\in X^f\;: \; x\in A\}$, then $\mu^f(Z')>1-\epsilon$.

Let $\delta':=\delta(\epsilon', 2\gamma \frac{\max(N_0,N)}{D_1})$. Take two points $(x,s),(x',s')\in Z'$, such that $x\neq x'$ and $d^f((x,s),(x',s'))<\delta'$.
It follows by definition of $d^f$ that $d(x,x')<\delta'$ and therefore by our assumptions there exist $M,L\geq 2\gamma\frac{\max(N_0,N)}{D_1}$ with $\frac{L}{M}\geq \kappa$, $p\in P$ and such that for all $n\in [M,M+L]$ either $|f^{(n)}(x)-f^{(n)}(y)-p|<\epsilon'$ or  for all $n\in [M,M+L]$, $|f^{(-n)}(x)-f^{(-n)}(y)-p|<\epsilon'$. We will consider the second case (the proof in the first case goes along the same lines).

Let us define
$$M':=\frac{f^{(-M)}(x)-s}{-\gamma}\;\;\;\text{and}\;\;\;
L':=\frac{f^{(-M-L)}(x)-f^{(-M)}(x)}{-\gamma}.$$
By~(\ref{adam1}) it follows that $L'=\frac{f^{(-L-M)}(x)-f^{(-M)}(x)}{-L}\frac{-L}{-\gamma}>
\frac{-LD_1}{-\gamma}>N$. Similarly, $\frac{f^{(-M)}(x)-s}{-\gamma}>\frac{f^{(-M)}(x)}{-\gamma}>\frac{MD_1}{\gamma}$, so $M'>N$. Moreover, since $(x,s)\in Z'$,  $s<\frac{2}{\epsilon^2}<ND_2\leq M D_1D_2/(2\gamma)$ (by the choice of $N$) and therefore
$$\frac{L'}{M'}\geq \frac{LD_1}{\gamma}\frac{-\gamma}{f^{(-M)}(x)-s}\geq \frac{LD_1}{MD_2}\geq \kappa'.$$
It follows by the properties of $M',L'\in \N$ that if $(x,s)\in Z'\subset Z$ we have
\begin{equation}\label{et2}\left|\frac{1}{L'}\sum_{k=M'}^{M'+L'}\chi_{X(\epsilon)}T^f_{-k\gamma}
-(1-\epsilon)\right|<\frac{\epsilon}{2}.\end{equation}
Take any $k\in [M',M'+L']$ such that $T^f_{-k\gamma}\in X(\epsilon)$ it follows that there exist a number $m_k\in [M,M+L]$ such that $T^f_{-k\gamma}(x,s)=(T^{m_k}x,-k\gamma+s-f^{(-m_k)}(x))$, where, by the fact that $T^f_{-k\gamma}\in X(\epsilon)$, $f^{(-m_k-1)}(x)+\frac{\epsilon}{8}<-k\gamma+s<f^{(-m_k)}(x)-\frac{\epsilon}{8}$. Using additionally the inequality $|s-s'|<\delta'$ we hence obtain
$$f^{(-m_k-1)}(x')\leq f^{(-m_k-1)}(x)+p-\epsilon'\leq f^{(-m_k-1)}(x)+p+\frac{\epsilon}{8}-\delta'<-k\gamma+s'+p.$$
A similar reasoning shows that
$$-k\gamma+s'+p<-k\gamma+s+p+\delta\leq f^{(-m_k)}(x)+p-\frac{\epsilon}{8}+\delta \leq f^{(-m_k)}(x').$$
Therefore, by the definition of the special flow, we have  $T^f_{-k\gamma+p}(x',s')=(T^{m_k}x',-k\gamma+s'+p-f^{(-m_k)}(x'))$. Consequently,
$$d^f(T^f_{-k\gamma}(x,s),T^f_{-k\gamma+p}(x',s'))=d^f((x,s),(x',s'))+|f^{(-m_k)}(x)-f^{(-m_k)}(x')-p|<\epsilon.$$
Now, the number of $k\in [M',M'+L']$ such that $T^f_{-k\gamma}\in X(\epsilon)$ is, by (\ref{et2}), at least $(1-\epsilon)L'$ and for any such $k$ we get that $d^f(T^f_{-k\gamma}(x,s),T^f_{-k\gamma+p}(x',s'))<\epsilon$. Hence
$$\frac{1}{L'}\left|\{k\in [M',M'+L']\;:\;d^f(T^f_{-k\gamma}(x,s),T^f_{-k\gamma+p}(x',s'))<\epsilon \}\right|>1-\epsilon.$$
This gives us the switchable $R(\gamma,P)$-property. \end{proof}

\vspace{1ex}

Note that if a flow $(T_t^f)_{t\in\R}$ is ergodic then the set of $\eta\in \R$ such that $T^f_\eta$ is not ergodic, is at most countable and therefore, as a  direct consequence of Proposition~\ref{cocy}, we get that
$(T^f_t)_{t\in\R}$ enjoys SWR-property.

\section{SWR-property for smooth special flows with singularities} \label{sec.proofs}
In this section we will use Proposition \ref{cocy} to prove SWR-property for  special flows given by the assumptions in Theorem~\ref{main} and Theorem~\ref{boun}. 
In all the sequel we assume $\{a_1,..,a_k\}$ are badly approximable  by $\a$ with a constant $C>1$ (see Definition \ref{gen.pos}).



\begin{lemma}\label{base2}  Let $s\in \N$ be such that $q_{s+1}>4Cq_s$ and $x\in \T$. Then 
$$\{x+j\alpha\}_{j=0}^{[\frac{q_{s+1}}{4C}]} \cap\bigcup_{i=1}^k\left[\frac{-1}{4Cq_s}+a_i,a_i+\frac{1}{4Cq_s}\right]\subset \{x+rq_s+i_0\alpha\}_{r=0}^{[\frac{q_{s+1}}{4Cq_s}]},$$
where $i_0\in \{0,...,q_s-1\}$ is such that $\rho(\{x+v\alpha\}_{v=0}^{q_s-1},\{a_i\}_{i=1}^k)=\rho(x+i_0\alpha,\{a_i\}_{i=1}^k)$. For finite sets $A,B\subset \T$, we use the notation $\rho(A,B)=\min_{a\in A,b\in B}\|a-b\|$.
\end{lemma}
Proof: Take any $0<j<q_s-1$, $j\neq i_0$. By \eqref{basic}, $x+j\alpha \notin \bigcup_{i=1}^k[\frac{-1}{2Cq_s}+a_i,a_i+\frac{1}{2Cq_s}]$ . It follows that for every $r=0,...,\max\{[\frac{q_{s+1}}{4Cq_s}],1\}$,
$$\|x+rq_s+j\alpha-(x+j\alpha)\|=\|rq_s\alpha\|\leq r\|q_s\alpha\|\leq \frac{1}{4Cq_s}.$$
Hence, we conclude by \eqref{basic}.\hfill $\square$

\vspace{1ex}

The following lemma is a simple consequence of the Denjoy-Koksma inequality.
\begin{lemma}\label{koksi} Let $h\in C^2(\T\setminus\{0\})$ be positive and decreasing on $(0,1)$ with $h'$ is increasing on $(0,1)$ and $\lim_{x\to 0^+}h(x)=\lim_{x \to 0^+}(-h'(x))=+\infty$. Denote by $c_0:=\inf_\T h$. Then for every $x\in\T$ and $s\in \N$ we have the following estimates:
$$-q_s\left(h(\frac{1}{2q_s})-c_0\right)-2h'(\frac{1}{2q_s})>h'^{(q_s)}(x)\geq h'(x+j\alpha) -q_s\left(h(\frac{1}{2q_s})-c_0\right)+2h'(\frac{1}{2q_s})$$
where $j\in \{0,...,q_s-1\}$ is such that $\min_{\ell\in\{0,...,q_s-1\}}|x+\ell\alpha|=x+j\alpha$.
\end{lemma}
\begin{proof} Fix $s\in \N$. Consider

$$
\bar{h}(x)=\begin{cases} 0,\; \text{if} \;x\in [0,\frac{1}{2q_s})\\
h'(x), \;\text{otherwise.}
\end{cases}$$

Then $\bar{h}\in BV(\T)$ and we use the Denjoy-Koksma inequality to obtain $|\bar{h}^{(q_s)}(x)-q_s\int_\T \bar{h}(t)d\lambda|<{\rm Var}\,\bar{h}$. But $\int_\T \bar{h}(t)d\lambda=h(\frac{1}{2q_s})-c_0$ and ${\rm Var}\,\bar{h} \leq-  2h'(\frac{1}{2q_s})$. We then finish since $h'^{(q_s)}(x)=\bar{h}^{(q_s)}(x)+h'(x+j\alpha)\chi_{[0,\frac{1}{2q_s}]}(x+j\alpha)$ and $h'<0$. \end{proof}

\begin{lemma}\label{upp} Let $f\in C^2(\T\setminus\{a_1,...,a_k\})$. Assume that for $i=1,...,k$, $\lim_{x\to a_i^+}\frac{f'(x)}{r_i(x-a_i)}$ and $\lim_{x\to a_i^-}\frac{f'(x)}{r_i(a_i-x)}$  exist and are finite, where $0\leq r_i\in C^2(\T\setminus\{0\})$ is decreasing on $(0,1)$ with $r_i'$ increasing on $(0,1)$.
Then there exists a constant $H>0$ such that
$$|f'(x)|<H \left(\sum_{i=1}^k -r_i'(x-a_i)-r_i'(a_i-x)\right),$$
for each $x\in\T$.
\end{lemma}
\begin{proof}
 By assumptions, there exists a constant $z_0>0$ such that for every $i=1,...,k$ and for every $x\in [-z_0+a_i,a_i)$ $|f'(x)|<-Kr_i'(a_i-x)$ and for every $x\in (a_i,a_i+z_0]$; $|f'(x)|<-Kr_i'(x-a_i)$ for some constant $K\geq 0$. \\
Moreover, since $f'\in C^1(\T\setminus \{a_1,...,a_k\})$; it follows that there exists a constant $R>0$ such that for every $x\in \T\setminus \bigcup_{i=1}^k[-z_0+a_i,a_i+z_0]$, $|f'(x)|<R$. Denote by $C_0:=\min_{i=1,...,k}\left|\sup_{\T}r_i'\right|$
Now, the constant $H:=2\max_{i=1,...,k}\{K,\frac{R}{C_0}\}$ satisfies the assertion of the lemma. \end{proof}

\subsection{Proof of Theorem~\ref{main}}

We may assume WLOG that $\sum_{i=1}^k(A_i-B_i)>0$.
Fix $1 \gg \epsilon>0$ and $N\in \N$. Let $d=\sum_{i=1}^k(A_i-B_i)-\min(\frac{1}{10},\frac{\sum_{i=1}^k(A_i-B_i)}{2})>0$. Define $\kappa=\kappa(\epsilon):=\frac{\epsilon m_0d}{64(d+1)Hk}$. By Lemma \ref{erg} for $\epsilon/2$ and $\kappa$, we get $N_0=N(\epsilon/2,\kappa)$ and a set $A:=A(\epsilon/2,\kappa)$ with $\lambda(A)>1-\frac{\epsilon}{2}$ such that (\ref{mne}) holds for $x\in A$. Define 

\begin{equation}\label{p;}P:=\left[-2(d+1),\frac{-dm_0}{32C}\right]\cup
\left[\frac{dm_0}{32C},2(d+1)\right],
\end{equation}
($C$ commes from Definition \ref{gen.pos} of badly approximable singularities).

In the sequel, we will assume $s$ is an integer sufficiently large $s \geq s_1$, $s_1=s_1(\eps,N)$ to be determined later, for now assume that $\kappa q_{s_1}>N$.

By assumptions 1. and 2. of Theorem \ref{main}, if $s \geq s_0$ and $s_0(\eps)$ is sufficiently large we will have 
\begin{equation}\label{mon}
\frac{|h'(\frac{x_s}{4C})|}{q_sh(\frac{1}{2q_s})}<\frac{\epsilon}{2},\;\;
\frac{x_s}{2C}q_sh\left(\frac{1}{2q_s}\right)>\frac{1}{\epsilon},\;\; h\left(\frac{1}{2q_s}\right)>8C,
\end{equation}
and $\sum_{s\geq s_0,s\notin K_\alpha}x_s q_s<\frac{\epsilon}{16k}$. Set $v_s:=\frac{x_s}{4C}$.

Let $$B_s:=\{x\in\T \;:\: x-q_s\alpha,...,x,...,x+(q_s-1)\alpha \notin \bigcup_{i=1}^k(-4v_s+a_i,a_i+4v_s)\}$$ 
and $Z':=\bigcap_{s\geq s_0,s\notin K_\alpha}B_s$. 

Define $Z:=Z'\cap A$. Observe that $\lambda(Z)\geq 1- \epsilon$ ($\lambda(B_s)\geq 1-16kv_sq_s $). 
Set $\delta:=\frac{1}{q_{s_1}h(\frac{1}{2q_{s_1}})}$. Consider $x,y\in Z$ with $0<\|x-y\|<\delta$ (we assume that $x<y$).

The following proposition  implies Theorem \ref{main} due to Proposition \ref{cocy}.

\begin{proposition} \label{ratner.log} Consider $x,y\in Z$ with $0<\|x-y\|<\delta$. Then there exists $p \in P$, $M,L\geq \kappa M \geq N$ such that either \eqref{posi}  holds for $n \in [M,M+L]$ or \eqref{nega} holds for $n \in [M,M+L]$. 
\end{proposition}

Proposition \ref{ratner.log} can be deduced from the following main result on the drift of the Birkhoff sums of a function with logarithmic like singularities.   Let $s:=s(x,y)$ ($s\geq s_0$) be unique such that

\begin{equation}\label{xoy}\frac{1}{q_{s+1}h(\frac{1}{2q_{s+1}})}< \|x-y\|\leq\frac{1}{q_sh(\frac{1}{2q_s})}.\end{equation}

We will assume that $q_{s+1}>2q_s$. If not, then  in (\ref{xoy}),  $\frac{m_0}{2q_sh(\frac{1}{2q_s})}<\|x-y\|$ and we repeat the considerations below in the time interval $[q_{s-1},q_s]$. In other words, in this case we will see the drift between $x$ and $y$ before time $q_s$.

\begin{proposition} \label{prop.drift.log} $ \ $  Consider $x,y\in Z$ as in \eqref{xoy}. 

\noindent {\bf Part a}  There exists  $n_0\in \{1,...,\max(\frac{q_{s+1}}{8Cq_s},1)\}$ satisfying 
\begin{equation}\label{epsi}
f^{(n_0q_s)}(x)-f^{(n_0q_s)}(y)\in P
\end{equation}
or 
\begin{equation}\label{epsi2}
f^{(-n_0q_s)}(x)-f^{(-n_0q_s)}(y)\in P.
\end{equation}

\noindent {\bf Part b} Let $X=T^{n_0q_s}x$ and $Y=T^{n_0q_s}y$ if \eqref{epsi} holds, and $X=T^{-(n_0q_s+1)}x$ and $Y=T^{-(n_0q_s+1)}y$ if \eqref{epsi2} holds. 
  For $n=1,...,[\kappa n_0q_s]+1$ the following holds
\begin{equation}\label{choi}
\text{A.}\;|f^{(n)}(X)-f^{(n)}(Y)|<\epsilon\;\;\;\; \text{or} \;\;\;\;\text{B.}\;|f^{(-n)}(X)-f^{(-n)}(Y)|<\epsilon.
\end{equation}
\end{proposition}

The rest of this section is devoted to the proof of Proposition \ref{prop.drift.log}. But before this we show how it implies Proposition \ref{ratner.log} and thus Theorem \ref{main}.  

\noindent {\emph{Proof of Proposition \ref{ratner.log}.}  Suppose \eqref{epsi} holds, the other case being similar. If A.  from (\ref{choi}) holds, set $M:=n_0q_s$, $L:=[\kappa M]+1$ and $p:=f^{(n_0q_s)}(x)-f^{(n_0q_s)}(y) \in P$. If B. holds, we set $M:=[\frac{n_0q_s}{1+\kappa}]$, $L:=[\kappa M]+1$ and $p:=f^{(n_0q_s)}(x)-f^{(n_0q_s)}(y)\in P$. Notice that in both cases $M,L\geq \frac{1}{2}  \kappa n_0q_s\geq \frac{1}{2}\kappa q_s\geq \frac{1}{2} \kappa q_{s_1}\geq N$. Finally, using $A.$ or $B.$ and the cocycle identity for the Birkhoff sums and the triangular inequality   shows that  for $n\in [M,M+L]$, $|f^{(n)}(x)-f^{(n)}(y)-p|< \eps$ for some $p\in P$.}
\carre

\subsubsection{Proof of Proposition \ref{prop.drift.log} Part a.}

For $m \in \N$, we will often use the following non resonance conditions of a pair of points $(x,y)$ with the singularities $\{a_1,\ldots,a_k\}$. 

\begin{equation} \label{nr1}     \bigcup_{j=0}^{q_m-1}T^j[x,y]\cap \bigcup_{i=1}^k[-2v_m+a_i,a_i+2v_m]=\emptyset     \end{equation}

\begin{equation}\label{100a}\bigcup_{j=0}^{\max(\left[\frac{q_{m+1}}{4C}\right],q_m)}T^j[x,y]\cap \bigcup_{i=1}^k[-v_m+a_i,a_i+v_m]=\emptyset.
\end{equation}
\begin{center}
\end{center}
\begin{equation}\label{100b} \bigcup_{j=0}^{\max(\left[\frac{q_{m+1}}{4C}\right],q_m)}T^{-j}[x,y]\cap \bigcup_{i=1}^k[-v_m+a_i,a_i+v_m]=\emptyset.
\end{equation}

\begin{lemma}\label{swit} Let $x,y\in \T$ as in  \eqref{xoy}. Then for every $m$ such that $s_0\leq m\leq s$, if we have at least one of the following  
\begin{enumerate}
    \item if $m\notin K_\alpha$ and \eqref{nr1} is satisfied 
\item if $m\in K_\alpha$ and  $q_{m+1}\geq 2q_m$,
\end{enumerate}
then we have at least one of \eqref{100a} or \eqref{100b}.

\end{lemma}

\begin{proof}
Assume $m\notin K_\alpha$. Since $m\geq s_0$ we have that $q_{m+1}\geq \frac{1}{x_m}\geq \frac{8k}{\epsilon}q_m\geq 16Cq_m$. Let $t_1\in [0,q_m-1]\cap \Z$ and $r_1\in \{1,...,k\}$ be such that $$\rho(\{x+j\alpha\}_{j=0}^{q_m-1},\{a_i\}_{i=1}^k)=\|x+t_1\alpha-a_{r_1}\|.$$
It follows by \eqref{basic} that for $t_1\neq j\in [0,q_m-1]\cap \Z$, we have
$$x+j\alpha\notin \bigcup_{r=1}^k\left[-\frac{1}{2Cq_m}+a_r,a_r+\frac{1}{2Cq_m}\right].$$
If $\|x+(-q_m+t_1)\alpha-a_{r_1}\|<\|x+t_1\alpha-a_{r_1}\|$,we will show \eqref{100a}, if we have the reversed inequality, we will show the validity of \eqref{100b}. Suppose 
\begin{equation}\label{chu}\|x+(-q_m+t_1)\alpha-a_{r_1}\|<\|x+t_1\alpha-a_{r_1}\|
\end{equation}
holds (the proof in the second case is analogous). For $j\notin \{t_1,t_1+q_m,...,t_1+[\frac{q_{m+1}}{4Cq_m}]q_m-1\}$, by Lemma \ref{base2} (for $i_0=t_1$), 
$$x+j\alpha\notin \bigcup_{i=1}^k\left[-\frac{1}{4Cq_m}+a_i,a_i+\frac{1}{4Cq_m}\right].$$ 
By (\ref{xoy}) and \eqref{mon}, for $j=0,...,[\frac{q_{m+1}}{4C}]-1$, 
\begin{equation}\label{xjy}\|(x+j\alpha)-(y+j\alpha)\|=\|x-y\|\leq v_s\leq \frac{\epsilon}{8kq_s}\leq \frac{1}{8Cq_s}\leq \frac{1}{8Cq_m},\end{equation}
hence \eqref{100a} follows for $j\notin \{t_1,t_1+q_m,...,t_1+[\frac{q_{m+1}}{4Cq_m}]q_m-1\}$. If $j=rq_m+t_1$ for some $r=0,...,[\frac{q_{m+1}}{4Cq_m}]-1$, then by \eqref{chu} (and using $t_1\leq q_m-1$) 
$$\rho(x+j\alpha,\{a_i\}_{i=1}^k)\geq r\|q_m\alpha\|+\rho(x+t_1\alpha,\{a_i\}_{i=1}^k)\geq \rho(x+t_1\alpha,\{a_i\}_{i=1}^k)\stackrel{\eqref{nr1}}{\geq}  2v_m.$$
 Therefore, since $\|(x+j\alpha)-(y+j\alpha)\|=\|x-y\|\leq v_s\leq v_m$ (see \eqref{xjy}), \eqref{100a}  follows for $j=rq_m+t$.\\
\indent If $m\in K_\alpha$ (recall that, by assumption $q_{m+1}\geq 2q_m$).\\
Denote by $u_m:=[\frac{q_{m+1}}{2}]$ and consider $x-u_m\alpha,...,x,...,x+(u_m-1)\alpha$ (the length of this orbit is smaller than $q_{m+1}$). By \eqref{basic}, there exists at most one $j_0\in [-u_m,u_m-1]$ such that 
$$x+j_0\alpha\in \bigcup_{i=1}^k\left[-\frac{1}{2Cq_{m+1}}+a_i,a_i+\frac{1}{2Cq_{m+1}}\right].$$
If $j_0<0$, we will show \eqref{100a}, if not, we will show the validity of \eqref{100b}. Suppose $j_0<0$. Then for every $j=0,...,u_m-1$,
\begin{equation}\label{cot}x+j\alpha\;,y+j\alpha\notin \bigcup_{i=1}^k[-v_m+a_i,a_i+v_m].\end{equation}
 Indeed, it follows that for $j=0,...,u_m-1$, $x+j\alpha\notin \bigcup_{i=1}^k[-\frac{1}{2Cq_{m+1}}+a_i,a_i+\frac{1}{2Cq_{m+1}}]$. But $m\in K_\alpha$, therefore $\frac{1}{2Cq_{m+1}}>\frac{x_m}{2C}=2v_m$. Moreover, by \eqref{xoy}, $\|x-y\|\leq v_s\leq v_m$. Therefore \eqref{100a} follows.
\end{proof}

\begin{lemma}\label{xcv}There exists $s'\in \N$ such that for every $s\geq s'$ and any points $x<y\in \T$ such that \eqref{nr1} is satisfied for $x,y$ and $m=s$, then 
\begin{equation}\label{A.}(d+1)q_sh\left(\frac{1}{2q_s}\right)\|x-y\|\geq f^{(q_s)}(x)-f^{(q_s)}(y)\geq dq_sh\left(\frac{1}{2q_s}\right)\|x-y\|.\end{equation}
\end{lemma}
\begin{proof} By \eqref{nr1} and \eqref{xoy}, $f^{(q_s)}$ is differantiable on $[x,y]$. Therefore, there exists $\theta\in [x,y]$ such that 
$$f^{(q_s)}(x)-f^{(q_s)}(y)=\|x-y\||f'^{(q_s)}(\theta)|.$$
It is enough to show that there exist $d>0$ and $s'\in \N$ such that for $s\geq s'$
$$(d+1)q_sh\left(\frac{1}{2q_s}\right)\geq|f'^{(q_s)}(\theta)|\geq dq_sh\left(\frac{1}{2q_s}\right).$$
For $s\in \N$, define
$$
\bar{f'_s}(\theta)=\begin{cases} 0,\; \text{if} \;\theta\in \bigcup_{i=1}^k[-\frac{1}{2q_s}+a_i,a_i+\frac{1}{2q_s}]\\
f'(\theta), \;\text{otherwise.}
\end{cases}$$
It follows that $\bar{f'_s}\in BV(\T)$ and \begin{equation}\label{ksk}f'^{(q_s)}(\theta)=\bar{f'_s}^{(q_s)}(\theta)+\sum_{i\in J_s} f'(\theta+j_i\alpha)+\sum_{i\in L_s}f'(\theta+l_i\alpha),
\end{equation}
 where $J_s=\{i \in [1,k] : \exists j_i\in\{0,...,q_s-1\}:\;\theta+j_i\alpha\in [-\frac{1}{2q_s}+a_i,a_i]\}$ and $L_s:=\{i \in [1,k]: \exists l_i\in\{0,...,q_s-1\}:\;\theta+l_i\alpha\in [a_i,a_i+\frac{1}{2q_s}]\}$. Note that for every $i  \in [1,k]
$ there exists at most one $ j_i\in\{0,...,q_s-1\}:\;\theta+j_i\alpha\in [-\frac{1}{2q_s}+a_i,a_i]$.

 We use the Denjoy-Koksma inequality to $\bar{f'_s}$, to get
\begin{equation}\label{kok2}q_s\int_{\T}\bar{f'_s}d\lambda-{\rm Var}(\bar{f'_s})\leq |\bar{f'_s}^{(q_s)}(\theta)|\leq q_s\int_{\T}\bar{f'_s}d\lambda+{\rm Var}(\bar{f'_s}).\end{equation}
We have 
\begin{equation}\label{29}\int_{\T}\bar{f'_s}d\lambda=\sum_{i=1}^k f(a_i+\frac{1}{2q_s})-f(a_i-\frac{1}{2q_s})\text{ and }{\rm Var}(\bar{f'_s})=2\sum_{i=1}^k\left(f'(a_i+\frac{1}{2q_s})+f'(a_i-\frac{1}{2q_s})\right),\end{equation} (if $s\in \N$ is sufficiently large). It follows by the assumptions on $f'$ and $h'$ and \eqref{nr1}, that for every $\epsilon>0$ there exists $s'=s'(\epsilon)$ such that for $s\geq s'$, we have for every $i=1,...,k$:
\begin{equation}\label{1a}|f'(\theta+j_i\alpha)|\leq (B_i+1)|h'(\theta+j_i\alpha)|\leq (B_i+1)|h'(\frac{x_s}{4C})|\leq \epsilon q_sh(\frac{1}{2q_s})\text{ for } i \in J_s.
\end{equation}
and similarly 
\begin{equation}\label{1b} |f'(\theta+l_i\alpha)|\leq  \epsilon q_sh(\frac{1}{2q_s})\text{ for } i \in L_s.
\end{equation}
On the other hand, by  l'Hospital's rule
\begin{equation}\label{1c}  \left((A_i+\epsilon)-(B_i-\epsilon)\right)h(\frac{1}{2q_s})\geq f(a_i+\frac{1}{2q_s})-f(a_i-\frac{1}{2q_s})\geq \left((A_i-\epsilon)-(B_i+\epsilon)\right)h(\frac{1}{2q_s}).
\end{equation}
\begin{equation}\label{1d} |f'(a_i+\frac{1}{2q_s})|+|f'(a_i-\frac{1}{2q_s})|\leq \left((A_i+1)+(B_i+1)\right)|h'(\frac{1}{2q_s})|\leq \epsilon q_sh(\frac{1}{2q_s}),
\end{equation}
(by $\frac{x_s}{4C}<\frac{1}{2q_s}$).

Now, using (\ref{ksk})-\eqref{1d}, we get

$$q_sh\left(\frac{1}{2q_s}\right)\left(\left(\sum_{i=1}^k(A_i-B_i)\right)-6k\epsilon\right)\leq|f'^{(q_s)}(\theta)|\leq q_sh\left(\frac{1}{2q_s}\right)\left(\left(\sum_{i=1}^k(A_i-B_i)\right)+6k\epsilon\right).$$ 
which allows us to conclude if we assume WLOG that $\eps$ is sufficiently small.
\end{proof}

We will assume in the sequel that $s_1\geq s'$ of Lemma \ref{xcv}.

\begin{lemma}\label{dri} Let $x,y\in \T$ satisfy \eqref{xoy}. Assume $x,y$ satisfy  \eqref{100a} with $m=s$, then 
there exists $n_0\in \{1,...,\max(\frac{q_{s+1}}{8Cq_s},1)\}$ such that \eqref{epsi} holds. 
Moreover,
\begin{equation}\label{wiem}n_0q_sh\left(\frac{1}{2q_s}\right)\leq\frac{2(d+1)}{d\|x-y\|}.
\end{equation}

If $x,y$ satisfy  satisfy  \eqref{100b} with $m=s$ then 
there exists $n_0\in \{1,...,\max(\frac{q_{s+1}}{8Cq_s},1)\}$ such that
\eqref{epsi2} holds.
for some $n_0\in \{1,...,\max(\frac{q_{s+1}}{8Cq_s},1)\}$ satisfying \eqref{wiem}.
\end{lemma}

\begin{proof}  We will use repeteadly Lemma \ref{xcv} for $x,y$ replaced by $x+rq_s\alpha, y+rq_s\alpha$ respectively, $r=0,1,...,\max([\frac{q_{s+1}}{4Cq_s}]-1,0)$ (note that by \eqref{100a} the points $x+rq_s,y+rq_s$ satisfy the assumptions of Lemma \ref{xcv}). If we fix $0\leq R\leq \max([\frac{q_{s+1}}{4Cq_s}],1)$, then using \eqref{A.} for $r=0,1,...,R$, summing up the obtained inequalities and using the cocycle identity, we obtain
\begin{equation}\label{dos}R\|x-y\|(d+1)q_sh\left(\frac{1}{2q_s}\right)>
f^{(Rq_s)}(x)-f^{(Rq_s)}(y)\geq R\|x-y\|dq_sh\left(\frac{1}{2q_s}\right).\end{equation}
Let $e_R:=f^{(Rq_s)}(x)-f^{(Rq_s)}(y)$. Then $e_{R+1}-e_R=f^{(q_s)}(x+Rq_s\alpha)-f^{(q_s)}(y+Rq_s\alpha)$, so in view of \eqref{dos}, \eqref{xoy}, we obtain
\begin{equation}\label{jed}|e_{R+1}-e_R|<d+1\;\;\text{ for } R=0,...,\max([\frac{q_{s+1}}{4Cq_s}]-1,0). 
\end{equation}
Moreover, by (\ref{dos}) and (\ref{xoy}), 
\begin{equation}\label{ven}e_{\max([\frac{q_{s+1}}{8Cq_s}],1)}\geq d\max(\frac{q_{s+1}}{8Cq_s},1)q_sh(\frac{1}{2q_s})\|x-y\|\geq
dm_0\max\left(\frac{1}{8C}-\frac{q_s}{q_{s+1}},
\frac{q_s}{q_{s+1}}\right)
\geq \frac{dm_0}{16C}.\end{equation}
Therefore, by \eqref{jed}, \eqref{ven} and \eqref{p;}, there exists $n_0\in \{1,...,\max([\frac{q_{s+1}}{8Cq_s}],1)\}$ such that
$$f^{(n_0q_s)}(x)-f^{(n_0q_s)}(y)=e_{n_0}\in P$$
Moreover, by \eqref{p;} and \eqref{dos} (for $R=n_0$), 
\begin{equation}\label{en0}n_0q_sh\left(\frac{1}{2q_s}\right)\leq\frac{2(d+1)}{d\|x-y\|}.\end{equation}

In case \eqref{100b} is satisfied instead of \eqref{100a}, we show \eqref{epsi2} in  exactly the same fashion as we did for \eqref{epsi}.

\end{proof}

\bigskip

We are ready now to finish the proof of Part a. of Proposition \ref{prop.drift.log}.  If $s\notin K_\alpha$, then by the fact that $x,y\in Z\subset B_s$, it follows that 1. in Lemma \ref{swit} is satisfied with $m=s$.
If $s\in K_\alpha$ then 2. in Lemma \ref{swit} is satisfied with $m=s$. Therefore we can use Lemma \ref{swit} for $x,y$ and $m=s$. Now, by Lemma \ref{xcv},    if \eqref{100a} holds we have \eqref{epsi}, if \eqref{100b} holds we have  \eqref{epsi2}.  Part a. of Proposition \ref{prop.drift.log} is settled, we turn now to Part b.


\subsubsection{Proof of Proposition \ref{prop.drift.log} Part b.}


\begin{lemma}\label{imba} Let $x,y\in \T$ satisfy \eqref{100a} for some $m\geq s_0$, then for every 
$\N\cup\{0\}\ni l\leq \max(\frac{q_{m+1}}{8Cq_m}-1,0)$ 
\begin{equation}\label{100a'}
\text{for every } n=0,...,(l+1)q_m,\;|f^{(n)}(x)-f^{(n)}(y)|<8kH\|x-y\|(l+1)q_mh(\frac{1}{2q_m})
\end{equation}
   If  $x,y\in \T$ satisfy \eqref{100b} for some $m\geq s_0$, then
\begin{equation}\label{100b'}
\text{for every } n=0,...,(l+1)q_m,\;|f^{(-n)}(x)-f^{(-n)}(y)|<8kH\|x-y\|(l+1)q_mh(\frac{1}{2q_m}).
\end{equation}
 \end{lemma}

\begin{proof}

 
We only give the proof of  the first case since the other is similar. For every\\
 $n=0,...,\max(\frac{q_{m+1}}{4C},q_m)$, there exists $\theta_n\in [x,y]$ such that $|f^{(n)}(x)-f^{(n)}(y)=|f'^{(n)}(\theta_n)|\|x-y\|$. Therefore, using Lemma \ref{upp}, for every $n=0,...,\max(\frac{q_{m+1}}{4C},q_m)$, we have
\begin{equation}\label{wet}
\left|f^{(n)}(x)-f^{(n)}(y)\right|\leq H\|x-y\|\left(
\sum_{i=1}^k-h'^{(n)}(\theta_n-a_i)-h'^{(n)}(a_i-\theta_n)\right).
\end{equation}
Moreover, by monotonicity of $h'$, for every $i=1,...,k$, $$-h'^{(n)}(\theta_n-a_i)\leq -h'^{(n)}(x-a_i)\text{ and }-h'^{(n)}(a_i-\theta_n)\leq -h'^{(n)}(a_i-y).$$ 
Since $-h'$ is positive, we get that \begin{equation}\label{aje}-h'^{(n)}(\theta_n-a_i)\leq -h'^{((l+1)q_m)}(x-a_i)\text{ and } -h'^{(n)}(a_i-\theta_n)<-h'^{((l+1)q_m)}(a_i-y).\end{equation}
 It follows by Lemma \ref{koksi}, (\ref{100a}) and (\ref{mon}) that for every $u=0,...,l$
$$\|x-y\||h'^{(q_m)}(T^{uq_m}x-a_i)|\leq \|x-y\|\left(q_mh(\frac{1}{2q_m})-h'(\frac{1}{2q_m})-h'(v_m)\right)\leq 4\|x-y\|q_mh(\frac{1}{2q_m}). $$
Hence, summing up over $u=0,...,l$, and using the cocycle identity, \eqref{wet} implies \eqref{100a'}.
This finishes the proof. \end{proof}

To prove Proposition \ref{prop.drift.log} \textbf{Part b.}, observe first that if $s_1$ is sufficiently large, and up to eventually changing $\kappa$ to $\kappa' = \frac{\kappa}{8C}$, one of two possibilities holds :  ${\bf 1.}$  There exists   $s_0\leq m\leq s$, $m \in K_\a$,  such that $ \kappa n_0 q_s \leq q_{m} \leq 8 C \kappa n_0 q_s$,  or ${\bf 2.}$  There exist $s_0\leq m\leq s$ and $l \geq 1$ such that  $ l q_m \leq \kappa n_0 q_s \leq (l+1) q_m \leq    \frac{q_{m+1}}{8C}$.

\noindent {\bf Case 1.}  $ \kappa n_0 q_s \leq q_{m} \leq 8 C \kappa n_0 q_s$ with $s_0\leq m\leq s$, $m \in K_\a$. Lemma \ref{swit} implies that either \eqref{100a} or \eqref{100b} holds for $T^{n_0 q_s}x,T^{n_0 q_s}y,m$. We then apply Lemma \ref{imba} to  $T^{n_0 q_s}x,T^{n_0 q_s}y,m$ with $l=0$, and according to whether we have \eqref{100a'} or  \eqref{100b'} we will get A. or B. of Proposition \ref{prop.drift.log} Part b. Indeed, suppose \eqref{100a'} holds then for $n =1, \ldots, [\kappa n_0 q_s]+1$, we have due to \eqref{en0}
\begin{align*} |f^{(n)}(x)-f^{(n)}(y)|&<8kH\|x-y\| q_mh(\frac{1}{2q_m}) \\
&<16C kH  \kappa n_0 q_s  \|x-y\|   h(\frac{1}{2q_m})\\
&<16C kH  \kappa    \frac{2(d+1)}{d} < \eps. \end{align*}

\medskip 

\noindent {\bf Case 2.} There exist $s_0\leq m\leq s$ and $l \geq 1$ such that  $ l q_m \leq \kappa n_0 q_s \leq (l+1) q_m \leq    \frac{q_{m+1}}{8C}$.

If $m\in K_\alpha$, then Lemma \ref{swit} implies that either \eqref{100a} or \eqref{100b} holds for $T^{n_0 q_s}x,T^{n_0 q_s}y,m$. 

 If $m\notin K_\alpha$, then we will first prove that  $T^{n_0 q_s}x,T^{n_0 q_s}y,m$ satisfy the hypothesis of Lemma \ref{imba}. Due to Lemma \ref{swit}, we just have to check \eqref{nr1} for  $T^{n_0 q_s}x,T^{n_0 q_s}y,m$ :
  \begin{equation}\label{kip}\bigcup_{j=0}^{q_m-1}T^j[T^{n_0q_s}x,T^{n_0q_s}y]\cap \bigcup_{i=1}^k[-2v_m+a_i,a_i+2v_m]=\emptyset.\end{equation}
Indeed, let $i_0$ and $r_1$ be such that $\rho(\{T^{n_0q_s}x+j\alpha\}_{j=0}^{q_m-1},\{a_i\}_{i=1}^k)= \|T^{n_0q_s}x+i_0\alpha-a_{r_1}\|$. It follows by \eqref{basic}, that for $i_0\neq j\in \{0,...,q_{m}-1\}$,
\begin{equation}\label{noq}T^{n_0q_s}x+j\alpha\notin \bigcup_{i=1}^k[-\frac{1}{2Cq_m}+a_i,a_i+\frac{1}{2Cq_m}].\end{equation}
Next, by the fact that $m\notin K_\alpha$ and $x\in B_m$ ($m\geq s_0$), we get that $\|x+i_0\a -a_{r_1}\|\geq 4v_m$, and therefore
$$\|x+i_0\a+n_0q_s\a -a_{r_1}\| \geq \|x+i_0\a -a_{r_1}\|-\|n_0q_s\a\| \geq 4 v_m-\frac{n_0}{q_{s+1}}\geq 4v_m-\frac{1}{8Cq_s}\geq 3 v_m,$$
(recall that $n_0\leq \frac{q_{s+1}}{8Cq_s}$)
and \eqref{kip} is thus proved ($\|T^{n_0q_s}x-T^{n_0q_s}y\|\stackrel{\eqref{xoy}}{<}v_m$).

We are now able to apply Lemma \ref{imba} to $T^{n_0 q_s}x,T^{n_0 q_s}y,m$ with $l$ such that  $l q_m \leq \kappa n_0 q_s \leq (l+1) q_m$. Now and as in case 1.,  if \eqref{100a'} holds we get A., if \eqref{100b'} holds we get B.


\subsection{Proof of Theorem \ref{boun}}
 We may assume WLOG that $A_k^2+B_k^2>0$. Let $C_k=\max(A_k,-B_k)>0$ \\
Let us define $$P:=[-12Hk (D_2+2),-\frac{C_kD_1^2}{16c}]\cup[\frac{C_kD_1^2}{16c},12Hk (D_2+2),]$$
where $H$ is from Lemma \ref{upp}, and $c$ is such that for every $s\in \N$ $q_{s+1}\leq cq_s$.
Fix $\frac{C_kD_1^2}{8c}>\epsilon>0$ and $N\in \N$. Let $\kappa:=\kappa(\epsilon)=\frac{\epsilon}{2(3D_2+2)kCH}$. We now use Lemma \ref{erg} to $f$ and $\T$ to get, $N_0=N(\epsilon/2,\kappa)$ and a set $A:=A(\epsilon/2,\kappa)$ with $\lambda(A)>1-\frac{\epsilon}{2}$ such that (\ref{mne}) holds for $x\in A$. Denote for every $s\in \N$, $x_s=\min_{i=1,...,v}x_{i,s}$ (note that, by definition, $x_s>\frac{1}{q_sh(\frac{1}{2q_s}})$). Let $s_0\in \N$ be such that for every $i=1,...,v$  and $\T\ni x\leq \frac{1}{q_{s_0}}$,
\begin{equation}\label{img}
\frac{g_i(x)}{h(x)}<\frac{\epsilon}{12kH(D_2+1)}\;\;\text{and}\;\;\frac{g_i'(x_{i,s})}{h'(\frac{1}{2q_s})}<\epsilon\min(\frac{1}{12Hk(D_2+1)},1)\;\;
 \text{for every}\; s\geq s_0,
\end{equation}
$\sum_{s\geq s_0}^{+\infty}q_sx_{i,s}<\frac{\epsilon}{8k}$ for every $i=1,\ldots,v$,
\begin{equation}\label{rtg} h(\frac{1}{2q_s})>6C,\;\;\text{for}\; s\geq s_0,
\end{equation}
and  for every $i=1,...k$
\begin{equation}\label{jnm}
\left|\frac{f'(x)}{h'(x-a_i)}\right|>\frac{A_i}{2}\;\text{for}\; x\in[a_i,a_i+\frac{1}{q_{s_0-4}}]\;\;\text{and}\;\;
\left|\frac{f'(x)}{h'(a_i-x)}\right|>\frac{B_i}{2},\;\text{for}\; x\in [-\frac{1}{q_{s_0-4}}+a_i,a_i].
\end{equation}
Define
$D_s:=\{x\in\T \;:\: x-q_s\alpha,...,x,...,x+(q_s-1)\alpha \notin (\bigcup_{i=1}^v(-x_s+a_i,a_i+x_s)\})$. It follows that $\lambda(B_s)\geq 1-4vq_sx_s\geq 1-4kq_sx_s $. Define $Z':=\bigcap_{s\geq s_0}B_s$. It follows that $\lambda(Z')>1-4k\sum_{s\geq s_0}q_sx_s\geq 1-\frac{\epsilon}{2}$. Now define $Z:=A\cap Z'$, $\lambda(Z)>1-\epsilon$.

Let $s'\geq s_0$ be such that $q_{s'-4}\geq \max\{\frac{1}{\kappa}N,N_0\}$. Define $\delta:=\frac{1}{q_{s'}h(\frac{1}{2q_{s'}})}$.

The following proposition  implies Theorem \ref{boun} due to Proposition \ref{cocy}.

\begin{proposition} \label{ratner.power} Consider $x,y\in Z$  with $0<\|x-y\|<\delta$. Then there exists $p \in P$, $M,L\geq \kappa M \geq N$ such that either \eqref{posi} or \eqref{nega} holds for $n \in [M,M+L]$. 
\end{proposition}

We can assume WLOG that $x<y$. Let $s:=s(x,y)$ be unique such that
\begin{equation}\label{stoj}\frac{1}{q_{s+1}h(\frac{1}{2q_{s+1}})}\leq \|x-y\|<\frac{1}{q_sh(\frac{1}{2q_s})}.\end{equation}

As in the precedent section, Proposition \ref{ratner.power} follows from 

\begin{proposition} \label{prop.drift.power} $ \ $  Consider $x,y\in Z$ as in \eqref{stoj}. 

\noindent {\bf Part a.}  There exists    $i_0 \in \{0,...,q_{s-2}-1\}$, such that 
\begin{equation}\label{epso}
|f^{(i_0)}(x)-f^{(i_0)}(y)| \in P
\end{equation}
or 
\begin{equation}\label{epso2}
|f^{(-i_0)}(x)-f^{(-i_0)}(y)| \in P.
\end{equation}

\noindent {\bf Part b.} Let $X=T^{i_0}x$ and $Y=T^{i_0}y$ if \eqref{epso} holds, and $X=T^{-i_0-1}x$ and $Y=T^{-i_0-1}y$ if \eqref{epso2} holds,
for $n=1,...,[\kappa i_0]+1$ the following holds
\begin{equation}\label{choo}
\text{A.}|f^{(n)}(X)-f^{(n)}(Y)|<\epsilon\;\;\;\; \text{or} \;\;\;\;\text{B.}\;|f^{(-n)}(X)-f^{(-n)}(Y)|<\epsilon.
\end{equation}
\end{proposition}

The rest of this section is devoted to the proof of Proposition \ref{prop.drift.power}.


Consider the orbit $x-q_{s-2}\alpha,...,x,...,x+(q_{s-2}-1)\alpha$ (the length of this orbit is smaller than $q_s$). It follows by \eqref{basic} that there exists at most one $t_s\in[-q_{s-2},q_{s-2}+1]$ such that
$x+t_s\alpha \in \bigcup_{i=v+1}^k[-\frac{1}{2Cq_{s}}+a_i,a_i+\frac{1}{2Cq_{s}}]$. Hence at least one of the following two holds :

\begin{equation} \label{nr_pos}     \bigcup_{j=0}^{q_{s-2}-1}T^j[x,y]\cap \bigcup_{i=v+1}^k[-\frac{1}{2Cq_{s}}+a_i,a_i+\frac{1}{2Cq_{s}}]=\emptyset     \end{equation}
or
\begin{equation} \label{nr_neg}    \bigcup_{j=1}^{q_{s-2}}T^{-j}[x,y]\cap \bigcup_{i=v+1}^k[-\frac{1}{2Cq_{s}}+a_i,a_i+\frac{1}{2Cq_{s}}]=\emptyset.     \end{equation}

The following Lemma directly implies the proof of Proposition \ref{prop.drift.power}.
\begin{lemma} \label{lemme.drift} If \eqref{nr_pos} then \eqref{epso} and \eqref{choo} hold. If \eqref{nr_neg} then \eqref{epso2} and \eqref{choo} hold.  \end{lemma}

\begin{proof}
We will suppose \eqref{nr_pos} holds, the proof of the other case being analogous.

\begin{lemmata}  For $n=0,...,q_{s-2}-1$,   
$$|f^{(n)}(x)-f^{(n)}(y)| \leq 2k(3D_2+2).$$
\end{lemmata}
\begin{proof} 

By the choice of $x,y\in Z$ and (\ref{stoj}) we have for every $i=1,\ldots,k$, $a_i\notin [x+j\alpha,y+j\alpha]$ with $j\in\{0,\ldots,q_{s-2}-1\}$.
It follows that for $n=0,...,q_{s-2}-1$, $\left|f^{(n)}(x)-f^{(n)}(y)\right|=
\left|f'(\theta_n)\right|\|x-y\|$, for some $\theta_n\in[x,y]$.
Hence, using Lemma \ref{upp} and (\ref{gi}), for every $n=0,...,q_{s-2}$ we have
\begin{multline}\label{thie}
\left|f^{(n)}(x)-f^{(n)}(y)\right|\leq\\ H\|x-y\|\left(\sum_{i=1}^v(-g'^{(n)}_i(\theta_n-a_i)-g'^{(n)}_i(a_i-\theta_n))+
\sum_{i=v+1}^k(-h'^{(n)}(\theta_n-a_i)-h'^{(n)}(a_i-\theta_n))\right).
\end{multline}
Let $\phi$ stand for $g_i$, $i=1,...,v$ and $h$. By the monotonicity of $\phi'$ on $(0,1)$ we obtain $-\phi'^{(n)}(\theta_n-a_i)\leq -\phi'^{(n)}(x-a_i), -\phi'^{(n)}(a_i-\theta_n)\leq -\phi'^{(n)}(a_i-y)$. Since $-\phi'>0$, \begin{equation}\label{try}-\phi'^{(n)}(x-a_i)-\phi'^{(n)}(a_i-y)\leq -\phi'^{(q_{s-2})}(x-a_i)-\phi'^{(q_{s-2})}(a_i-y).\end{equation}
Using Lemma \ref{koksi} (applied to $x-a_i$, where $j_i\in [0,q_{s-2}]-1$ is unique such that $x+j_i\alpha\in [a_i,a_i+\frac{1}{2q_{s-2}}]$), we obtain
$$ \|x-y\|\left(-\phi'^{(q_{s-2})}(x-a_i)\right)\leq \|x-y\|\left(q_{s-2}\phi(\frac{1}{2q_{s-2}})-2\phi'(\frac{1}{2q_{s-2}})-\phi'(x+j_i\alpha)\right).$$
Consider $i\in E$. It follows that for $n=0,...,q_{s-2}-1$ we have $x+n\alpha,y+n\alpha\notin \bigcup_{i=v+1}^k[-\frac{1}{3Cq_s}+a_i,a_i+\frac{1}{3Cq_s}]$, because by (\ref{rtg}) and (\ref{stoj}), $\|x-y\|<\frac{1}{6Cq_s}$ and $x+n\alpha \notin \bigcup_{i=v+1}^k[-\frac{1}{2Cq_s}+a_i,a_i+\frac{1}{2Cq_s}]$. In this case, bo monotonicity of $h'$ , $-h'(x+j_i\alpha)<-h'(\frac{1}{2Cq_s})$ and therefore by (\ref{dd}) and (\ref{stoj})
\begin{multline}\label{ah} \|x-y\|\left(-h'^{(q_{s-2})}(x-a_i)\right)\leq \|x-y\|\left(q_{s-2}h(\frac{1}{2q_{s-2}})-2h'(\frac{1}{\frac{1}{2}\frac{1}{q_{s-2}}})-h'(\frac{1}{\frac{1}{2C}\frac{1}{q_s}})\right)\leq\\
\frac{q_{s-2}h(\frac{1}{2q_{s-2}})+2D_2q_{s-2}h(\frac{1}{2q_{s-2}})+D_2q_sh(\frac{1}{2q_s})}{q_sh(\frac{1}{2q_s})}\leq
\frac{3D_2q_sh(\frac{1}{2q_s})+q_{s-2}h(\frac{1}{2q_{s-2}})}{q_sh(\frac{1}{2q_s})}\leq 3D_2+1.
\end{multline}

Similarly (replacing $\frac{1}{2Cq_s}$ by $\frac{1}{3Cq_s}$); we obtain $\|x-y\|\left(-h'^{(q_{s-2})}(a_i-y)\right)<3D_2+1$.\\

For $i\in F$, by the fact that $x,y\in Z$, monotonicity of $g_i'$ and the choice of $s_0$, it follows that $-g_i'(x+i_j\alpha-a_i)\leq -g_i'(x_{i,s})\leq g_i'(x_s)\leq \epsilon h'(\frac{1}{2q_s})$. Therefore, using (\ref{img}), we get
\begin{multline} \label{sgh}\|x-y\|\left(-g_i'^{(q_{s-2})}(x-a_i)\right)\leq \|x-y\|\left(q_{s-2}g_i(\frac{1}{2q_{s-2}})-2g_i'(\frac{1}{\frac{1}{2}\frac{1}{q_{s-2}}})-\epsilon h'(\frac{1}{\frac{1}{2C}\frac{1}{q_s}})\right)\leq\\
\|x-y\| \left(\epsilon q_{s-2}h(\frac{1}{2q_{s-2}})-\epsilon h'(\frac{1}{\frac{1}{2}\frac{1}{q_{s-2}}})-\epsilon h'(\frac{1}{\frac{1}{2C}\frac{1}{q_s}})\right)\leq \epsilon(3D_2+1),
\end{multline}
in the last inequality we use the last estimation in (\ref{ah}). Similarly we prove $\|x-y\|\left(-g_i'^{(q_{s-2})}(a_i-y)\right)<\epsilon(3D_2+1)$. Therefore using \eqref{thie} and the computations above, for $n=0,...,q_{s-2}-1$,   $$|f^{(n)}(x)-f^{(n)}(y)|<H(\epsilon 2v(3D_2+1)+2(k-v)(3D_2+1))\leq 2k(3D_2+2),$$ by the choice of $\epsilon$.
\end{proof}

\begin{lemmata}
 There exists  $i_0 \in \{0,...,q_{s-2}-1\}$, such that 
$|f^{(i_0)}(x)-f^{(i_0)}(y)| \geq \frac{A_kD_1^2}{4c}$.
\end{lemmata}
\begin{proof} Since $q_{s-2}-q_{s-4}>q_{s-4}+1$, there exists  $i_0\in [q_{s-4},q_{s-2}-2]$ such that $x+i_0\alpha\in [a_{k},a_k+\frac{1}{q_{s-4}}]$. We have assumed that $A_k^2+B_k^2>0$. Suppose additionally $A_k\geq-B_k$ (if $A_k\leq -B_k$ then we replace $[a_{k},a_k+\frac{1}{q_{s-4}}]$ by $[-\frac{1}{q_{s-4}}+a_k,a_k]$). We claim that
$$|(f^{(i_0+1)}(x)-f^{(i_0+1)}(y))-(f^{(i_01)}(x)-f^{(i_0)}(y))|>\frac{A_kD_1^2}{2c}.$$
Indeed, the LHS of this inequality is  equal to $|f(x+i_0\alpha)-f(y+i_0\alpha)|=|f'(\theta_{i_0})|\|x-y\|$, for some $\theta_{i_0}\in [x+i_0\alpha,y+i_0\alpha]$. Now, by (\ref{stoj}), $\theta_{i_0}\in [a_k,a_k+\frac{1}{q_{s-4}}+\frac{1}{q_sh(\frac{1}{2q_s})}]\subset [a_k,a_k+\frac{2}{q_{s-4}}]$. By (\ref{jnm}), monotonicity of $h'$, (\ref{dd}) twice (for $s$ and $s+1$) and (\ref{stoj})
\begin{multline}|f'(\theta_{i_0})|\geq \frac{A_k}{2} |h'(\theta_{i_0}-a_k)|\geq \frac{A_k}{2} |h'(\frac{2}{q_{s-4}})|\geq \frac{A_k}{2} |h'(2c^4\frac{1}{q_s})|\geq\\
 \frac{A_k}{2}D_1q_sh(\frac{1}{2q_s})\geq \frac{A_k}{2}D_1^2 \frac{q_{s+1}}{c}h(\frac{1}{2q_{s+1}})\geq \frac{A_kD_1^2}{2c}\frac{1}{\|x-y\|};\end{multline}
and the claim follows. Therefore, one of the numbers $|f^{(i_0+1)}(x)-f^{(i_0+1)}(y)|$ or $|f^{(i_0)}(x)-f^{(i_0)}(y)|$ is at least $\frac{A_kD_1^2}{4c}$.
\end{proof}
As a consequence of the above lemmas, we obtain that at least one of the numbers 
$f^{(i_0+1)}(x)-f^{(i_0+1)}(y)$
$f^{(i_0)}(x)-f^{(i_0)}(y)$ belongs to the set $P$, and \eqref{epso} is proved. The next result is the proof of \eqref{choo}.
\begin{lemmata} The following hold:
\begin{equation}\label{alt}
|f^{(n)}(T^{i_0+1}x)-f^{(n)}(T^{i_0+1}y)|<\epsilon \;\;\text{for all}\;0\leq n\leq \kappa(i_0+1),\end{equation}
\begin{equation}\label{alt2}|f^{(-n)}(T^{i_0}x)-f^{(-n)}(T^{i_0}y)|<\epsilon\;\; \text{for all}\; 0\leq n\leq \kappa(i_0+1).
\end{equation}
\end{lemmata}

\begin{proof}
First we show \eqref{alt}. Select (a unique) $m\in \N$ such that $q_m\geq \kappa(i_0+1)\geq q_{m-1}$. By \eqref{basic} applied to $T^{i_0}(x)$, by the choice of $i_0$ it follows that
$$\{T^{i_0}x,...,T^{i_0}x+(q_m-1)\alpha\}\cap \bigcup_{i=v+1}^k[-\frac{1}{2Cq_m}+a_i,a_i+\frac{1}{2Cq_m}]=\{T^{i_0}x\}.$$
Therefore, using the same arguments which lead (\ref{thie}) we obtain (cf. (\ref{try})) for $n=0,...,\kappa(i_0+1)$
\begin{multline}\label{try2}\left|f^{(n)}(T^{i_0}x)-f^{(n)}(T^{i_0}y)\right|\leq H\|x-y\|
(\sum_{i=1}^v(-g'^{(q_m)}_i(T^{i_0+1}x-a_i)-g'^{(q_m)}_i(a_i-T^{i_0+1}y))+\\
\sum_{i=v+1}^k-h'^{(q_m)}(T^{i_0+1}x-a_i)-h'^{(q_m)}(a_i-T^{i_0+1}y)).\end{multline}
Then for $i\in E$, again by repeating that lead to (\ref{ah}) we obtain
$$\|x-y\|\left(-h'^{(q_{m})}(T^{i_0+1}x-a_i)\right)\leq
\frac{q_{m}h(\frac{1}{2q_{m}})+3D_2q_{m}h(\frac{1}{2q_{m}})}{q_sh(\frac{1}{2q_s})}\leq
\frac{(3D_2+1)q_mh(\frac{1}{2q_m})}{q_sh(\frac{1}{2q_s})}.$$
But $q_m\leq c\kappa(i_0+1)<c\kappa q_{s-2}$, thus (by the monotonicity of $h$)
$\|x-y\|\left(-h'^{(q_{m})}(T^{i_0+1}x-a_i)\right)\leq (3D_2+1)\frac{c\kappa q_{s-2}}{q_s}=\frac{\epsilon}{4Hk}$, by the definition of $\kappa$. Similarly (replacing $\frac{1}{2Cq_m}$ by $\frac{1}{3Cq_m}$), we obtain $\|x-y\|\left(-h'^{(q_{m})}(a_i-T^{i_0+1}y)\right)<\frac{\epsilon}{2Hk}$. \\
If $i\in F$, then using monotonicity of $g'$, the choice of $i_0$ and $m$, the fact that $x,y\in Z$ and (\ref{stoj}), we get $-g'^{(q_m)}_i(T^{i_0+1}x-a_i)\leq -g'^{(q_{s-1})}_i(x-a_i)$ and $-g'^{(q_m)}_i(a_i-T^{i_0+1}y)\leq -g'^{(q_{s-1})}_i(a_i-y)$. We proceed, repeating what lead to (\ref{sgh}) (with $q_{s-1}$ instead of $q_{s-2}$) and using (\ref{stoj}) and (\ref{img})
$$\|x-y\|\left(-g'^{(q_{s-1})}_i(x-a_i)\right)\leq \frac{q_{s-1}g'_i(\frac{1}{2q_{s-1}})-g_i'(\frac{1}{2q_{s-1}})-g_i'(x_s)}{q_{s}h(\frac{1}{2q_s})}\leq
\frac{\epsilon}{4kH}.$$
Similarly $\|x-y\|\left(-g'^{(q_{s-1})}_i(a_i-x)\right)<\frac{\epsilon}{4kH}$.
Using this and (\ref{try2}) we get\\ $|f^{(n)}(T^{i_0+1}x)-f^{(n)}(T^{i_0+1}y)|<\epsilon$, which yields the first case of (\ref{alt}). To handle the second case, notice that
$$\{T^{i_0}x-(q_m-1)\alpha,...,T^{i_0}x\}\cap \bigcup_{i=v+1}^k[-\frac{1}{2Cq_m}+a_i,a_i+\frac{1}{2Cq_m}]=\{T^{i_0}x\}.$$
We now proceed as before to obtain first $|f^{(-n)}(T^{i_0}x)-f^{(-n)}(T^{i_0}y)|=\|x-y\|\left|f'^{(n)}(\theta_n)\right|$
with $\theta_n\in [T^{i_0}x-(n-i_0)\alpha,T^{i_0}y-(n-i_0)\alpha]$ and then estimating above by
\begin{multline}H\|x-y\|
(\sum_{i=1}^v(-g'^{(q_m)}_i(T^{i_0}x-(q_m-1)\alpha-a_i)-g'^{(q_m)}_i(a_i-(T^{i_0}y-(q_m-1)\alpha)))+\\
\sum_{i=v+1}^k-h'^{(q_m)}(T^{i_0}x-(q_m-1)\alpha-a_i)-h'^{(q_m)}(a_i-(T^{i_0}y-(q_m-1)\alpha))).\end{multline}
We conclude exactly in the same way as in the first case.
\end{proof}
We proceed to the proof of Lemma \ref{lemme.drift} in the case \eqref{nr_pos} is satisfied. If $f^{(i_0+1)}(x)-f^{(i_0+1)}(y)\in P$, then \eqref{alt} gives A. in \eqref{choo}; if  $f^{(i_0)}(x)-f^{(i_0)}(y)\in P$, then \eqref{alt2} gives B. in \eqref{choo}.   The proof of Lemma \ref{lemme.drift} is thus completed since the case  where \eqref{nr_neg} is satisfied is analogous. 

 This finishes the proof of Proposition \ref{prop.drift.power}, thus of  Theorem \ref{boun}.  \end{proof}

\section{Absence of weak Ratner's property}\label{abs}
In this section, we will prove Theorem \ref{aby}. Let $f$ be as in Theorem \ref{aby}; for simplicity we assume that $\int_{\T}f=1$. Let $c>1$ be such that for every $s\in \N$, $q_{s+1}\leq c q_s$. Recall that $C>1$ is a constant from Definition \ref{gen.pos} (such a constant exists, since $k=1$ in our case);  
we may assume that $C>c$.

Fix any compact $P\subset \R\setminus\{0\}$. We will prove that for any $t_0\in \R$, $(T_t^f)_{t\in\R}$ does not have $R(t_0,P)$ property. For simplicity of the notations we will assume that $t_0=1$. Let 
\begin{equation}\label{d}  
 d>c^{1-\gamma} \text{ be such that }P\subset\left[-\frac{|\gamma|d}{4},-\frac{100c}{d}\right]\cup\left[\frac{100c}{d},\frac{|\gamma|d}{4}\right].
\end{equation}
Let $\epsilon,\kappa>0$ sufficiently small, smallness  that will be determined in the course of the proof. We use Lemma \ref{erg} for $Tx=x+\alpha$, to $\epsilon, 3\kappa^2$ to get a set $A\subset \T$, $\lambda(A)>1-\epsilon$ and $N_0\in\N$, such that (\ref{mne}) holds for $x\in A$ and $n\geq N_0$. Let $N>\max\left(2N_0,\frac{1}{\epsilon^2\kappa^2}\right)$.

We will hereafter assume  that $(T_t^f)_{t\in\R}$ has the $R(t_0,P)$ property (see Definition \ref{wrs}) and obtain a contradiction. Thus, assume there exist a set $Z\subset X^f$ with $\lambda^f(Z)>1-\epsilon$ and $0<\delta<\epsilon$ such that for every $(x,s),(y,s')\in Z$ with $d^f((x,s),(y,s'))<\delta$, there exist $M,L\geq N$ with $\frac{L}{M}\geq \kappa$ and $p\in P$ such that
\begin{equation}\label{smi}\frac{1}{L}\left|\{n\in [M,M+L]\;:\; d^f(T^f_n(x,s),T^f_{n+p}(y,s'))<\epsilon\}\right|>1-\epsilon.\end{equation}
Consider 
\begin{equation}\label{ohn}V:=\{(x,s)\in Z\;:\; x\in A,\;0\leq s<\frac{1}{\epsilon^2}\}.\end{equation} 
It follows that $\lambda^f(V)>1-4\epsilon$.

The contradiction will come from the following two Propositions, the first one of which is a consequence of \eqref{smi} and \eqref{ohn}.

\begin{proposition} \label{fan2}  Let $(x,s),(y,s')\in V$ with $d^f((x,s),(y,s'))<\delta$.
Then there exists an interval  $I=[M',M'+L']$ such that $M'\geq \frac{N}{2}$, $\frac{L'}{M'}\geq \frac{a\kappa}{6}$ ($a=a(t_0)>0$ is a constant obtained in Lemma \ref{nwc}), there exists $p\in P$ and there exists $m\in \Z$ such that 

\begin{multline}\label{fan}
\|x-y-m\alpha\|<\epsilon\text{ and for every }n\in [M',M'+L'],\;|f^{(n)}(x)-f^{(n+m)}(y)-p|<2\epsilon. 
\end{multline}

\end{proposition} 
\begin{remark}{\em \hspace{0.8cm} For $p\in P$, $n\in \N$, two points $(x,s),(y,s')\in V$ are called {\em $p,n$-close} if $d^f(T^f_n(x,s),T^f_{n+p}(y,s'))<\epsilon$. Then $(x,s),(y,s')$ have the WR-property (see \eqref{smi}), if there exists a time interval $[M,M+L]$, such that they are $p,n$-close for a proportion $1-\epsilon$ of n's in $[M,M+L]$. In general, the set on which the points are $p,n$-close, can be any subset of $[M,M+L]$. Proposition \ref{fan2} says that in our context the property actually holds on a full interval of integers. This is what happens also in the original case of horocycle flows, where, once the point are drifted after time $R$, they stay drifted for time $\eps R$.
}
\end{remark}

\begin{proposition} \label{fini} There exists a set $W_0\subset \T$ such that $\lambda(W_0)>c_0(d)$ ($c_0=c_0(d)>0$ being a constant depending only on $d$), and a number $0<\delta_0<\delta$ such that for every $x\in W_0$
\begin{multline}\label{naw} \text{for every}\; M\geq \frac{N}{2},\;\text{every }k\in \Z\;\text{ such that }\|x-(x+\delta_0)-k\alpha\|<\epsilon\text{ and every }p\in P,\\
\text{ if } I=[M,M+T]\text{ is such that for every }n\in I,\; |f^{(n)}(x)-f^{(n+k)}(x+\delta_0)-p|<2\epsilon \\ \text{ then }\frac{T}{M}<\frac{a\kappa}{10}.
\end{multline}
\end{proposition} 

\begin{remark}\label{expl} {\em 
The points $x \in W_0$ go too close to the singularity under iteration by $R_\a$, so that points of the form $(x,s), (x+\delta_0,s)$  split far apart  before they get separated by a  distance in $P$ (Lemma \ref{lm2} below). In other words, these points do not have the 'natural' WR-property that consists of a controlled drift starting from the first time the points split. To make sure these points cannot display the WR-Property in the future  $\delta_0$ is chosen in such a way, that if for large $M$ $T^f_M(x,s),T^f_M(x+\delta_0,s)$ become close, then nevertheless $d^f(T^f_M(x,s),T^f_M(x+\delta_0,s))\gg \frac{1}{M^{1-\gamma}}$, and Lemma \ref{len} then precludes the WR-property (see Lemma \ref{lm1} below). 
  }
\end{remark}

Before we prove these propositions we will see how they imply Theorem \ref{aby}.

\begin{proof} [Proof of Theorem \ref{aby}] Take $x \in W_0$ and $s>0$ such that $(x,s), (x+\delta_0,s) \in V \times V$, which is possible since the measure of $V$ is arbitrarily close to $1$ if $\eps$ is sufficiently small.  By Propostion \ref{fini},  $(x,s), (x+\delta_0,s)$ satisfy \eqref{naw}, hence they don't satisfy \eqref{fan}, a contradiction.
\end{proof}

\subsection{Proof of Proposition \ref{fan2}}

\begin{lemma}\label{len} 


Let $x,y\in \T$ and let $I$ be an integer interval such that for every $n\in I$, $|f^{(n)}(x)-f^{(n)}(y)|<\eta$ (where $\eta$ is a sufficiently small number). Then $|I|<2c\eta^{1+\gamma} \|x-y\|^{\frac{-1}{1-\gamma}}$.
 \end{lemma}
\begin{proof} We assume that $x<y$. Let $s\in \N$ be unique such that
\begin{equation}\label{xyk}\frac{1}{q^{1-\gamma}_{s+1}}\leq\|x-y\|<\frac{1}{q_s^{1-\gamma}}.
\end{equation}
 Denote $I=[a,b]\cap \Z$ with $a,b\in \Z$. Then, by the cocycle identity, the fact that $a\in I$, for $n\in \Z$, we have
\begin{equation}\label{on}|f^{(n)}(x)-f^{(n)}(y)|\geq |f^{(n-a)}(T^ax)-f^{(n-a)}(T^ay)|-\eta.\end{equation}
Let $k\in \N$ be unique such that 
\begin{equation}\label{two}q_{k+1}\geq \frac{2\eta^{1+\gamma}q_{s+1}}{c}>q_k.\end{equation} 
We will show that there exists $n_0\in [0,q_{k+1}]$ such that \begin{equation}\label{thr}|f^{((n_0+a)-a)}(T^ax)-f^{((n_0+a)-a)}(T^ay)|=|f^{(n_0)}(T^ax)-f^{(n_0)}(T^ay)|>2\eta.
\end{equation} 
This, by (\ref{on}), gives $|f^{(n_0+a)}(x)-f^{(n_0+a)}(y)|>\eta$ and therefore $n_0+a\notin I$. It follows that $|I|\leq q_{k+1}\leq cq_k<2\eta^{1+\gamma}q_{s+1}\stackrel{\eqref{xyk}}{\leq} 2c\eta^{1+\gamma}\|x-y\|^{\frac{-1}{1-\gamma}}$ which completes the proof. Now, we show (\ref{thr}). By (\ref{two}) and $\eta$ sufficiently small, we have $s\geq k$.\\
Note that there exist $n_1\in [0,q_{k+1})$ such that $T^ax+n_1\alpha\in[0,\frac{1}{q_{k+1}}]$. By (\ref{xyk}) and the fact that $k+1\leq s+1$, we obtain $T^ay+n_1\alpha\in [0,\frac{2}{q_{k+1}}]$. Therefore  
\begin{multline}\left|(f^{(n_1+1)}(T^ax)-f^{(n_1+1)}(T^ay))-(f^{(n_1)}(T^ax)-f^{(n_1)}(T^ay))\right|=\\
|f(T^ax+n_1\alpha)-f(T^ay+n_1\alpha)|=|f'(\theta)|\|x-y\|,\end{multline}
for some $\theta\in [T^ax+n_1\alpha,T^ay+n_1\alpha]\subset [0,\frac{2}{q_{k+1}}]$.
Thus, by the monotonicity of $f'$ and (\ref{two}) 
 $$|f'(\theta)|\|x-y\|\geq |\gamma|(\frac{2}{q_{k+1}})^{-1+\gamma}\frac{1}{q^{1-\gamma}_{s+1}}=
|\gamma|\left(\frac{q_{k+1}}{2q_{s+1}}\right)^{1-\gamma}\geq |\gamma|(\frac{\eta^{1+\gamma}}{c})^{1-\gamma}\geq 4\eta,$$
the last inequality by the fact that $\eta$ is small enough. Therefore at least one of the numbers, $|f^{(n_1+1)}(T^ax)-f^{(n_1+1)}(T^ay)|$, $|f^{(n_1)}(T^ax)-f^{(n_1)}(T^ay)|$ is bigger than $2\eta$; we set $n_0$ either $n_1$, or $n_1+1$ to obtain (\ref{thr}).\end{proof}

The following lemma translates \eqref{smi} into a property on the Birkhoff sums above $R_\a$ of the ceiling function $f$. 

\begin{lemma}\label{nwc}Let $(x,s),(y,s')\in V$ with $d^f((x,s),(y,s'))<\delta$.
 There exist $M_0,L_0\geq \frac{N}{2}$ with $\frac{L_0}{M_0}\geq \frac{\kappa}{2}$ such that
\begin{equation}\label{mo} \frac{1}{L_0}\left|\{r\in [M_0,M_0+L_0]\;:\;\exists m_r\in\Z,\;\text{s.t}\; |x-y-m_r\alpha|<\epsilon\;\text{and}\; |f^{(r)}(x)-f^{(r+m_r)}(y)-p|<2\epsilon\}\right|>a.
\end{equation}
\end{lemma}

\begin{proof}

Assume WLOG that $x<y$.
Let $n\in [M,M+L]$ and $r_n$ be unique such that $f^{(r_n)}(x)\leq n+s<f^{(r_n+1)}(x)$. We will show that \begin{equation}\label{wre}\frac{1+\kappa^2}{1-\kappa^2}M\geq r_M\geq \frac{1}{1+\kappa^2}M-2.\end{equation}  Indeed, first we show that $r_M>N_0$. Indeed, if not, using Lemma \ref{erg} to $\frac{N}{2}$ (we have $\frac{N}{2}\geq N_0$)
$$M\leq M+s<f^{(r_M+1)}(x)<f^{(\frac{N}{2})}(x)<(1+\kappa^2)\frac{N}{2}<N,$$ 
a contradiction. Secondly, by the fact that $(x,s)\in V$ (hence $s<\frac{1}{\epsilon^2}$) and $r_M\geq N_0$, using Lemma \ref{erg} to $r_M$ and the definition of $N$ ($N\geq \frac{1}{\epsilon^2\kappa^2}$), we get ($M\geq N$)
$$(1-\kappa^2)r_M<f^{(r_M)}(x)\leq M+s<(1+\kappa^2)M$$
and $M\leq M+s< f^{(r_M+1)}(x)\leq(1+\kappa^2)(r_M+1)$. Now, \ref{wre} follows. \\
\indent Analogously we prove that
\begin{equation}\label{cie}\frac{1}{1+\kappa^2}(M+L)-2\leq r_{M+L}\leq \frac{1+\kappa^2}{1-\kappa^2}(M+L).\end{equation}
 Set $M_0:=r_M, L_0=r_{M+L}-r_M$. It is easy to prove using (\ref{wre}) and (\ref{cie}) that $M_0,L_0\geq \frac{N}{2}$ and $\frac{L_0}{M_0}\geq \frac{\kappa}{2}$. Moreover, since $f\geq c_\gamma>0$, there exists a constant $a=a(t_0,\gamma)>0$ such that for every $n\in [M,M+L]$, $|r_{n+1}-r_n|\leq\frac{1}{2a}$. It follows that the number of different $r_n\in [M_0,M_0+L_0]$ is at least $2aL_0$.\\
Let $n\in [M,M+L]$ be such that $d^f(T^f_n(x,s),T^f_{n+p}(y,s'))<\epsilon$. By (\ref{smi}), there are at least $(1-\epsilon)L$ of such $n\in [M,M+L]$. By the  definition of $d^f$ and $T^f$, there exist $r_n\in [M_0,M_0+L_0]$ and $m_n\in \N$ such that $$|(x+r_n\alpha)-(y+m_n\alpha)|<\epsilon\;\; \text{and}\;\; |f^{(r_n)}(x)-f^{(m_n)}(y)-p|<2\epsilon.$$
We set $m_r=m_r(n):=m_n-r_n\in \Z$ to get $|x-y-m_r\alpha|<\epsilon$ and $|f^{(r_n)}(x)-f^{(r_n+m_r)}(y)-p|<2\epsilon$. It follows that the number of different $r_n\in [M_0,M_0+L_0]$ is at least $2a(1-\epsilon)L_0$ and hence (\ref{mo}) follows.

\end{proof}

\begin{proof} [Proof of Proposition \ref{fan2}] 

 Denote by
$$U:=\{r\in [M_0,M_0+L_0]\;:\;\exists_{m_r} \;\;|x-y-m_r\alpha|<\epsilon\; \text{and}\; \;|f^{(r)}(x)-f^{(r+m_r)}(y)-p|<2\epsilon\}.$$ It follows by (\ref{mo}) that $|U|\geq aL_0$. Let us choose in the integer interval $[M_0,M_0+L_0]$ disjoint subintervals $I_1=[a_1,b_1],...,I_l=[a_l,b_l]$ such that $U=I_1\cup...\cup I_l$ and for every $i=1,...,l$ there exists $m_i\in \Z$ such that $|x-y-m_i\alpha|<\epsilon$ and for $r\in I_i$,  $|f^{(r)}(x)-f^{(r+m_i)}(y)-p|<2\epsilon$. Moreover we assume that for every $i=1,...,l$, $I_i$ is maximal in the sense that $|f^{(h_i)}(x)-f^{(h_i+m_i)}(y)-p|\geq 2\epsilon$ for $h_i=a_i-1,b_i+1$. \\
 We will show that there exists $i=1,...,l$ such that 
\begin{equation}\label{ii}|I_i|\geq \frac{aL_0}{3}.\end{equation}
This will obviously finish the proof of \eqref{fan} with $M'=a_i, L'=|I_i|, \text{ and }m=m_i\in \Z$.\\
Let us show \eqref{ii}. If $l\leq 2$ there is nothing to prove. Assume $l\geq 3$.\\

 Notice that $U$ is the set of $n'$s such that $(x,s),(y,s')\in V$ are $p,n$-close. The next lemma implies that between any two disjoint integer intervals $I_{j}, I_{j+1}\subset U$, on which $(x,s),(y,s')$ are $p,n$-close, there will be an integer interval $J_j$ much longer than $I_j$, such that for any $n\in J_j$, $(x,s),(y,s')$ are not $p,n$-close.

\begin{lemmata}\label{beet} Let $i\in \{2,...,l-1\}$. There exist an interval $[c_i,d_i]=J_i\subset [M_0,M_0+L_0]$ such that for any $r\in J_i$, $4C^3>|f^{(r)}(x)-f^{(r+m_i)}(y)-p|\geq 2\epsilon$, $c_i>b_{i-1}$, $d_i<a_{i+1}$ and $|J_i|\geq \frac{|I_i|}{4C^3\epsilon^{1+\gamma}}$ (here $C>0$ commes from  \eqref{basic}).
\end{lemmata}
 Lemmata \ref{beet} will give (\ref{ii}). Indeed, by the definition of $J_i$ and $I_i$, it follows that for $i,j=2,...,l-1$ with $j\neq i-1,i,i+1$
$$J_i\cap I_i=\emptyset\;\;\;\text{and}\;\;\;J_i\cap J_j=\emptyset.$$

Hence, $\sum_{i=2}^{l-1} |J_i| \leq 3 L_0$, and 
$$|I_2\cup...\cup I_{l-1}|<12C^3\epsilon^{1+\gamma}L_0<\frac{aL_0}{3}$$
Therefore, by the fact that $|U|>aL_0$, we have $|I_1\cup I_l|>\frac{2aL_0}{3}$ and consequently, $|I_w|\geq \frac{aL_0}{3}$ for at least one of  $w=0$ or $w=1$. \\
 Hence to obtain (\ref{ii}) we just need to prove Lemma \ref{beet}.

\begin{proof} [Lemma \ref{beet}] Let $v\in \N$ be unique such that \begin{equation}\label{xy2a}\frac{1}{q^{1-\gamma}_{v+1}}<\|x-(y+m_i\alpha)\|\leq\frac{1}{q^{1-\gamma}_v}.
\end{equation} 
Consider $n\in I_i=[a_i,b_i]$. We have 
\begin{multline}
2\epsilon\geq |f^{(n)}(x)-f^{(n+m_i)}(y)-p|=\\
\left|(f^{(a_i)}(x)-f^{(a_i+m_i)}(y)-p)+(f^{(n-a_i)}(T^{a_i}x)-f^{(n-a_i)}(T^{a_i+m_i}y))\right|\geq\\
\left||f^{(n-a_i)}(T^{a_i}x)-f^{(n-a_i)}(T^{a_i+m_i}y)|-2\epsilon\right|.
\end{multline}
Hence, for $n\in I_i$, $|f^{(n-a_i)}(T^{a_i}x)-f^{(n-a_i)}(T^{a_i+m_i}y)|<4\epsilon$. It follows now by Lemma \ref{len} applied to $\eta=4\epsilon$, the points $T^{a_i}x, T^{a_i+m_i}y$ and $s=v$ that 
\begin{equation}\label{fiv} |I_i|<2(4\epsilon)^{1+\gamma}\|x-(y+m_i\alpha)\|^{\frac{-1}{1-\gamma}}\leq 2(4\epsilon)^{1+\gamma}q_{v+1}.
\end{equation}
 Consider integer intervals $K_i=[a_i-q_{v-2},a_i]$ and $L_i=[a_i,a_i+q_{v-2}]$. It follows by \eqref{basic} with $k=1$ and for the point $T^{a_i}x$, similarly to the proof of Theorem \ref{boun}, that there exist at most one $t_0\in K_i\cup L_i$ such that $T^{a_i}x+t_0\alpha\in[-\frac{1}{2Cq_v},\frac{1}{2Cq_v}]$. Assume $t_0<0$. Then we consider $L_i$. Moreover, we may assume that $\epsilon^\frac{\gamma}{1-\gamma}>2cC$, and therefore, using (\ref{xy2a}) we obtain $\frac{1}{2Cq_v}\leq \frac{1}{q_v^{1-\gamma}}$. It follows that for $n\in [0,q_{v-2}]$, $0\notin [T^{a_i+n}x,T^{a_i+m_i+n}y]$. Hence,

\begin{equation}\label{sign}n\rightarrow sign(f(T^{a_i+n}x)-f(T^{a_i+m_i+n}x)) \text{ is constant for }n=0,...,q_{v-2},
\end{equation}
(it may happen that $T^{m_i}y<x$). It follows that $|f^{(n)}(T^{a_i}x)-f^{(n)}(T^{a_i+m_i}y)|_{n=0}^{q_{v-2}}$ is increasing.
Hence, for $q_{v-2}>n\geq b_i+1$ we have $$f^{(n-a_i)}(T^{a_i}x)-f^{(n-a_i)}(T^{a_i+m_i}y)>4\epsilon.$$

Moreover, by  Lemma \ref{koksi} (the RHS of the inequality) to $h=f$, $x=\theta$ (where $f^{(q_{v-2})}(T^{a_i}x)-f^{(q_{v-2})}(T^{a_i+m_i}y)=f'^{(q_{v-2})}(\theta)\|x-y-m_i\alpha\|$) and $s=v-2$, we obtain
$$|f^{(n)}(T^{a_i}x)-f^{(n)}(T^{a_i+m_i}y)|<
|f^{(q_{v-2})}(T^{a_i}x)-f^{(q_{v-2})}(T^{a_i+m_i}y)|\leq 9C^{2}+4<2C^3,$$
(if necessery, to get the last inequality, we consider a bigger $C$).
We set $J_i=[b_i+1,a_i+q_{v-2}]$ ($b_i\leq a_i+2(4\epsilon)^{1+\gamma}q_{v+1}$, by (\ref{fiv})). It follows that for $n\in J_i$, we have by cocycle identity
\begin{multline}2\epsilon=4\epsilon-2\epsilon<
|f^{(n-a_i)}(T^{a_i}x)-f^{(n-a_i)}(T^{a_i+m_i}y)|-|f^{(a_i)}(x)-f^{(a_i+m_i)}(y)-p|\leq\\ |f^{(n)}(x)-f^{(n+m_i)}(y)-p|\leq\\ 
|f^{(a_i)}(x)-f^{(a_i+m_i)}(y)-p|+|f^{(n-a_i)}(T^{a_i}x)-f^{(n-a_i)}(T^{a_i+m_i}y)|\leq 2\epsilon+2C^3<4C^3.\end{multline}
Now, by (\ref{fiv}) 
\begin{multline}d_i-c_i=|J_i|\geq q_{v-2}-2(4\epsilon)^{1+\gamma}q_{v+1}=
2(4\epsilon)^{1+\gamma}q_{v+1}\left(\frac{q_{v-2}}{2(4\epsilon)^{1+\gamma}q_{v+1}}-1\right)\geq
|I_i|\frac{1}{4C^3\epsilon^{1+\gamma}},
\end{multline}
since $\frac{q_{v-2}}{q_{v+1}}\geq\frac{1}{c^3}\geq \frac{1}{C^3}$ and $\epsilon$ is small enough.\\
Suppose $n\in J_i\cap I_{i+1}$. If $m_{i+1}=m_i$, then immediately we have a contradiction. If $m_{i+1}\neq m_i$ then 
\begin{multline}|f^{(n)}(x)-f^{(n+m_{i+1})}(y)-p|=|(f^{(n)}(x)-f^{(n+m_i)}(y)-p)+f^{(m_{i+1}-m_i)}(T^{n+m_i)y}\geq\\
|m_{i+1}-m_i|\inf_{\T}f-4C^3>2\epsilon,
\end{multline}
since $|m_{i+1}-m_i|$ is of order $\frac{1}{\epsilon}$ because both $\|x-y-m_{i+1}\alpha\|$ and $\|x-y-m_{i}\alpha\|$ are close to zero. Hence $d_i<a_{i+1}$ and Lemma \ref{beet} has been proved.
\end{proof}
This finishes the proof of Proposition  \ref{fan2}.

\end{proof}

\subsection{Proof of Proposition \ref{fini}}

\begin{lemmata}\label{delta}
Fix a number $0<\zeta\leq \frac{|\gamma|}{50}$.  For every $v\in \N$, $v\geq v_0$ and $v_0$ sufficiently large, there exists $0<\delta^v_0<\delta$, satisfying
\begin{equation}\label{del'}\delta^v_0\in [\frac{1}{q_v^{1-\gamma}},\frac{2c}{q_v^{1-\gamma}}],\end{equation}

 such that for every $\N\ni|k|\geq \eps^{-\frac{1}{2}}$,
\begin{equation}\label{del3} \|\delta^v_0-k\a\|\geq\frac{1}{|k|^{1+\zeta}}.
\end{equation}
\end{lemmata}
\begin{lemmata}\label{w0w} For every  $w\geq w_0$, $w_0$ sufficiently large,  there exists a set 
\begin{equation}\label{wo}W^w_0\subset A\cap (A-\delta_0), 
\end{equation}
with $\lambda(W^w_0)\geq c_0$ ($c_0(d)$ will be specified in the proof),  such that the following holds for $x\in W^w_0$ and $y:=x+\delta^w_0$:
    
   \begin{equation}\label{sho1} |f^{(n)}(x)-f^{(n)}(y)|<100dc\text{ for every }n=0,...,q_{w-2}\end{equation}
	there exists $i_0\in\{ 0,...,q_{w-l}-1\}$ such that	
 
\begin{equation}\label{moz}x+i_0\alpha\in [\frac{1}{2dq_w},\frac{1}{dq_w}],\end{equation}
for some $\N\ni l\geq 1$ depending on $w$, to be specified later,
 
\begin{equation}\label{sho}0<f^{(n)}(x)-f^{(n)}(y)<\frac{100c}{d}\text{ for $n\leq i_0$  and }f^{(n)}(x)-f^{(n)}(y)>\frac{|\gamma|d}{2} \text{ for $i_0<n<wq_w$}.\end{equation}

\end{lemmata}

Before we proof the above Lemmas, let us first show, how they imply Proposition \ref{fini}.

Let $w\in \N$ be such that 

\begin{equation}\label{wu}\frac{1}{2}w\geq \frac{1}{\kappa^2}.
\end{equation}
 Denote $W_0:=W_0^w$ and $\delta_0:=\delta^w_0$. By definition,
\begin{equation}\label{del}\delta_0\in [\frac{1}{q_w^{1-\gamma}},\frac{2c}{q_w^{1-\gamma}}].\end{equation}

\begin{lemma}\label{lm1} We have that (\ref{naw}) holds for $k\neq 0$.
\end{lemma}

\begin{proof}

Fix any $M\geq \frac{N}{2}$, any $p\in P$ and any $k\neq 0$ such that $\|x-y-k\alpha\|<\epsilon$.\\

 
Let $I=[M,M+R]$ be such that for $n\in [M,M+R]$, we have $|f^{(n)}(x)-f^{(n+k)}(y)-p|<2\epsilon$.\\
Note that since $\|x-y-k\a\|<\eps$ and $\|x-y\|<\delta<\eps$, then $|k|\geq \frac{1}{2c^3\epsilon}$.
By the cocycle identity and the triangle inequality, this implies that for every $n\in [0,R]$ 
$$|f^{(n)}(x+M\a)-f^{(n)}(y+M\a+k\a)|\leq 4\epsilon.$$
Therefore, be Lemma \ref{len}, \begin{equation}\label{bsd}R<2c (4\eps)^{1+\gamma}\|x-y+k\alpha\|^{\frac{-1}{1-\gamma}}\stackrel{\eqref{del3}}\leq 2c (4\eps)^{1+\gamma}|k|^\frac{1+\zeta}{1-\gamma}\leq |k|^\frac{1+\zeta}{1-\gamma}.\end{equation}
On the other hand, since $x \in W_0$ and $|f^{(n)}(x)-f^{(n+k)}(y)-p|<2\epsilon$ we get that 
\begin{equation}2\kappa M>|f^{(M)}(x)-f^{(M)}(y)|\geq |f^{(k)}(T^My)|-2\epsilon-p\geq |k|\inf_\T f-2\epsilon-p>\frac{1}{2}|k|\inf_\T f\geq \frac{\inf_\T f}{4c^3\eps}.
\end{equation}

Hence, and using \eqref{sho1} ($\epsilon$ is small enough), 
\begin{equation}\label{mf} M>\max(q_{w-2},\frac{1}{2}|k|).
\end{equation} 
Hence, by \eqref{bsd} and \eqref{mf},
$$\frac{R}{M}\leq  \frac{|k|^\frac{1+\zeta}{1-\gamma}}{\max(q_{w-2},\frac{|k|}{2})}\leq  2
q_{w-2}^{\frac{\gamma-\zeta}{1-\gamma}} 
\stackrel{\eqref{wu}}{<}\frac{a\kappa}{10}.$$

 This proves (\ref{naw}) for $k\neq 0$.
\end{proof}

\begin{lemma}\label{lm2}
We have that \eqref{naw} holds for $k=0$.
\end{lemma}
\begin{proof}
Fix any $M\geq \frac{N}{2}$, $p\in P$ and let $I=[M,M+R]$ be such that for every $n\in I$, $|f^{(n)}(x)-f^{(n)}(y)-p|<2\epsilon$. By \eqref{sho} and (\ref{d}), for every $n\in [0,wq_w]$ and every $p\in P$, $|f^{(n)}(x)-f^{(n)}(y)-p|>2\epsilon$. Therefore, for $M\leq \frac{1}{2} wq_w$, \eqref{naw} holds.
 It follows by Lemma \ref{len} applied to $x$ and $y$ (we have $\|x-y\|=\delta_0\stackrel{\eqref{del}}{\geq}\frac{1}{q_w^{1-\gamma}}$), that $R=|I|<2(4\epsilon)^{1+\gamma}\|x-y\|^{\frac{-1}{1-\gamma}} \leq 2(4\epsilon)^{1+\gamma} q_{w}<q_w$. Therefore, by (\ref{wu}), for $M>\frac{1}{2}wq_w$,
$$\frac{R}{M}\leq \frac{q_w}{\frac{1}{2}wq_w}<\frac{a\kappa}{10}.$$ 
So, (\ref{naw}) holds for $k=0$.
\end{proof}

We thus proved \eqref{naw} in Proposition \ref{fini}. Let us now complete the proof by proving Lemmatas \ref{delta} and \ref{w0w}.

\begin{proof}[Proof of Lemmata \ref{delta}] 
Fix $v\in \N$. To simplify the notations, we will write $\delta_0$ instead of $\delta_0^v$.  
 Given $u\in \N$, set 
\begin{equation}\label{bu}B_u:=\{\eta\in \T\;:\; d\left(\eta,\{i\alpha\}_{i=-q_u}^{q_u-1}\right)\geq \frac{1}{2u^2q_u}\}.\end{equation}
Let $t\in \N$ be unique such that
\begin{equation}\label{tw}\frac{1}{q_{t+2}}<\frac{1}{q_v^{1-\gamma}}\leq \frac{1}{q_{t+1}}.\end{equation}
 
Let $c_1=4c$, then $q_{t+1}\geq 4c q_{t-c_1}$ (since $t$ depends on $v$ which is sufficiently large, $t-c_1>v-4$ by \eqref{tw}).  
\begin{equation}\label{del2}\left[\frac{1}{q_{t+1}},\frac{2}{q_{t+1}}\right]\cap \left(\bigcap_{i\geq t-c_1}B_i\right).
\end{equation}
 We will show below that this set is not empty. Now, let $\delta_0$ be any number in this set. This, by the definition of $B_i$, will give \eqref{del'} and \eqref{del3}. Indeed, \eqref{del'} follows from \eqref{tw} and \eqref{del2}. To show \eqref{del3}, note that for $|k|<q_{t-c_1}$, 
$$\|\delta_0-k\alpha\|\geq \|k\a\|-\delta_0\stackrel{\eqref{del'},\eqref{tw}}{\geq} \frac{1}{2c|k|}-\frac{1}{q_{t+1}}\geq \frac{1}{2c|k|}-\frac{1}{4cq_{t-c_1}}\geq \frac{1}{|k|^{1+\zeta}},$$
since $|k|\geq \eps^{-\frac{1}{2}}$.
If
$|k|\geq q_{t-c_1}$, let $\ell$ be unique, such that $q_{\ell+1}>|k|\geq q_\ell$. By definition of  $ B_{\ell+1}$
 $$\|\delta_0-k\alpha\|\geq \sup_{|i|\leq q_{\ell+1}}\|\delta_0-i\alpha\|\geq \frac{1}{(\ell+1)^2q_{\ell+1}}\geq \frac{1}{|k|}^{1+\zeta},$$
where the last inequality follows if $v$ is sufficiently large (then by \eqref{tw}, $t$ is large, so $|k|$ is large and therefore $\ell$ is large).

\begin{lemmata}
\begin{equation}\label{indu} \left[\frac{1}{q_{t+1}},\frac{2}{q_{t+1}}\right]\cap \left(\bigcap_{i\geq t-c_1}B_i\right)\neq \emptyset.
\end{equation}
\end{lemmata}

\begin{proof} The proof goes by induction. We will show that for every $k\geq t-c_1$ there exists a closed interval $E_k\subset \left[\frac{1}{q_{t+1}},\frac{2}{q_{t+1}}\right]\cap \left(\bigcap_{i\geq t-c_1}^{k}B_i\right)$ and $E_{k+1}\subset E_k$. Moreover, we will show (for the induction purpose) that for every $k\geq t-c_1$, $|E_k|\geq \frac{1}{c^{c_1+2}q_k}$. Indeed, we have $\min_{-q_{t-c_1}+1\leq i\leq q_{t-c_1}-1}\|i\alpha\|\geq \frac{1}{2q_{t-c_1}}\geq \frac{2}{q_{t+1}}$. Set $E_{t-c_1}:=[\frac{1}{q_{t+1}},\frac{2}{q_{t+1}}-\frac{1}{2(t-c_1)^2q_{t-c_1}}]$. It follows by the fact that $t-c_1\geq v-4$ and $v$ is sufficiently large (by taking a bigger $c$, we may assume that $c>2$) that $$|E_{t-c_1}|=\frac{1}{q_{t+1}}-\frac{1}{2(t-c_1)^2q_{t-c_1}}\geq 
\frac{1}{c^{c_1+2}q_{t-c_1}}\left(c-\frac{1}{c}\right)\geq \frac{1}{c^{c_1+2}q_{t-c_1}}.
$$

 Moreover, since $t-c_1$ is sufficiently large, $\frac{1}{q_{t+1}}>\frac{1}{2(t-c_1)^2q_{t-c_1}}$. Moreover, $c_1\geq 4$ and therefore $\frac{2}{q_{t+1}}-\frac{1}{2(t-c_1)^2q_{t-c_1}}\leq \min_{i\in\{-q_{t-c_1},...,q_{t-c_1}-1\}}\|i\a\|-\frac{1}{2(t-c_1)^2q_{t-c_1}}$. Therefore, by definition of $E_{t-c_1}$ and $B_{t-c_1}$,

$$E_{t-c_1}\subset B_{t-c_1}= \bigcap_{i=t-c_1}^{t-c_1}B_i.$$

Suppose that for some $k\geq t-c_1$ we have a closed interval $E_k\subset \left[\frac{1}{q_{t+1}},\frac{2}{q_{t+1}}\right]\cap \left(\bigcap_{i= t-c_1}^{k}B_i\right)$ such that  $|E_k|\geq \frac{1}{c^{c_1+2}q_k}$.
It follows that
$$E_k\cap\{i\alpha\}_{i=-q_k}^{q_k-1}=\emptyset.$$ 
Let $E_{k+1}\subset E_k$ be the longest closed subinterval (in $E_k$) such that
\begin{equation}\label{lon} E_{k+1}\cap B^c_{k+1}=\emptyset. 
\end{equation}
It follows that $E_{k+1}\subset E_k\subset\left[\frac{1}{q_{t+1}},\frac{2}{q_{t+1}}\right]$, and by (\ref{lon}), $E_{k+1}\subset \left(\bigcap_{i= t-c_1}^{k+1}B_i\right)$ ($E_{k+1}\subset E_k$). It remains to prove that $|E_{k+1}|\geq \frac{1}{c^{c_1+2}q_{k+1}}$. To do this note that
$$|E_{k+1}|\geq \frac{|E_k|}{\left[4|E_k|q_{k+1}+1\right]}-\frac{1}{(k+1)^2q_{k+1}}.$$ 
Indeed, $|E_k\cap\{i\alpha\}_{-q_{k+1}}^{q_{k+1}-1}|\leq 4|E_k|q_{k+1}$ and around each point of the form $i\alpha$, $i=-q_{k+1},...,q_{k+1}-1$, we discard an interval of length $\frac{1}{(k+1)^2q_{k+1}}$ (see \eqref{bu}, for $u=k+1$). We use the induction assumption, the fact that $k+1\geq t-c_1\geq v-4$ (and $v$ is sufficiently large) to obtain 
$$\frac{|E_k|}{\left[4|E_k|q_{k+1}+1\right]}-\frac{1}{(k+1)^2q_{k+1}}\geq\frac{1}{c^{c_1+2}q_{k+1}}.$$
Hence (\ref{indu}) is proved. \end{proof}
The proof of Lemmata \ref{delta} is thus finished.
\end{proof}

\begin{proof}[Proof of Lemmata \ref{w0w}]
We will determine the set $W^w_0$. To simplify notation, we will write $W_0$ instead of $W^w_0$, the dependence on $w$ will be clear from the context.
Let $l\in \N$ be such that
\begin{equation}\label{quwu} \left(\frac{q_{w-l}}{q_w}\right)^{-\gamma+1}\leq\frac{1}{d}<\left(\frac{q_{w-l+1}}{q_w}\right)^{-\gamma+1},
\end{equation}
since $w$ is large, $l\geq 1$.
Set 
$$W_{0,1}:=\{x\in \T\;:\; \{x,x+\alpha,...,x+(wq_w-1)\alpha\}\cap[-\frac{2c}{q_w^{1-\gamma}},\frac{2c}{q_w^{1-\gamma}}]=\emptyset\},$$
$$W_{0,2}:=\{x\in \T\;:\; \{x,x+\alpha,...,x+(q_{w-l}-1)\alpha\}\cap[\frac{1}{2dq_w},\frac{1}{dq_w}]\neq\emptyset\}.$$
We have $\lambda(W_{0,1})\geq 1-wq_w\frac{4c}{q_w^{1-\gamma}}=1-\frac{4cw}{q_w^{-\gamma}}$.\\
As $l\geq 1$, for $i=0,...,q_{w-l}-1$, the sets $T^i[\frac{1}{2dq_w},\frac{1}{dq_w}]$ are pairwise disjoint. Therefore, by (\ref{quwu}) 
$$\lambda(W_{0,2})=q_{w-l}\frac{1}{2dq_w}\geq \frac{1}{c} \frac{q_{w-l+1}}{q_w}\frac{1}{2d}\geq \frac{1}{2c}\left(\frac{1}{d}\right)^{\frac{1}{1-\gamma}}\frac{1}{d}=
\frac{1}{2c}\left(\frac{1}{d}\right)^{\frac{2-\gamma}{1-\gamma}}.$$
Now we set
\begin{equation}\label{w0} W_0:=W_{0,1}\cap W_{0,2}.
\end{equation}
Since $w\geq w_0$ is sufficiently large, $\lambda(W_0)\geq c_0$, where $c_0=c_0(d)>0$.
We may assume that
 $$W_0\subset A\cap (A-\delta_0),$$
if not we take $W_0:=W_0\cap A\cap (A-\delta_0)$ and use the fact that $\lambda(A)>1-\epsilon$, and $\delta_0<\delta$ is small. This gives \eqref{wo}.
Note that by \eqref{w0} and the definition of $W_{0,2}$, \eqref{moz} follows.
Let us show \eqref{sho1}.

 By \eqref{moz}, we have 
\begin{equation}\label{moz2}
\{x,x+\alpha,...,x+(q_{w-2}-1)\alpha\}\cap [0,\frac{1}{2dq_w}]=\emptyset.
\end{equation}

Therefore, using (\ref{del}) and $\|x-y\|=\delta_0$, we have for every $i=0,...,q_{w-2}-1$
\begin{equation}\label{moz3}
T^i([x,y])\cap [0,\frac{1}{6dq_w}]=\emptyset.
\end{equation}
By \eqref{moz3}, for $i=0,...,q_{w-2}-1$, $0\notin [x+i\a,y+i\a]$, and therefore
$|f^{(i)}(x)-f^{(i)}(y)|\leq |f^{(q_{w-2})}(x)-f^{(q_{w-2})}(y)|$.
Therefore and by (\ref{moz3}), (\ref{del}), monotonicity of $f'$, (\ref{moz3}), Lemma \ref{koksi} (to $h=f'$, and some $\theta\in [x,y]$), it follows that for $n\leq q_{w-2}$ 
\begin{multline}|f^{(n)}(x)-f^{(n)}(y)|\leq |f^{(q_{w-2})}(x)-f^{(q_{w-2})}(y)|\leq \|x-y\||f'^{(q_{w-2})}(\theta)|\leq \frac{2c}{q_{w}^{1-\gamma}}26dq_{w}^{1-\gamma}\leq\\
 100dc,\end{multline}
and \eqref{sho1} follows.\\
\\

Now we will show \eqref{sho}.

Moreover, for $x\in W_0\subset W_{0,1}$, and $y:=x+\delta_0$, $0\notin [x+i\a,y+i\a]$ for $i=0,...,wq_w-1$
\begin{equation}\label{seq}(f^{(n)}(x)-f^{(n)}(y))_{n=0}^{wq_w-1}\text{ is an increasing sequence.} 
\end{equation}

By the fact that $i_0<q_{w-l}$ and (\ref{moz})
it follows that $\|x+j\alpha-0\|>\frac{1}{2q_{w-l}}$ for $j\neq i_0$, $0\leq j\leq q_{w-l}-1$. Moreover, since $w\geq w_0$ is large enough, $\|x-y\|=\delta_0\leq\frac{2c}{q^{1-\gamma}_{w}}<\frac{1}{6q_w}$, and we obtain
 \begin{equation}\label{xj}
\|x+j\alpha\|,\|y+j\alpha\|>\frac{1}{6q_{w-l}}
\text{ for }j=0,...,q_{w-l}-1,j\neq i_0.
\end{equation}
 Moreover, by (\ref{seq}), for $n\leq i_0$

\begin{equation}\label{ela}0<f^{(n)}(x)-f^{(n)}(y)\leq (f^{(q_{w-l})}(x)-f^{(q_{w-l})}(y))-(f(x+i_0\alpha)-f(y+i_0\alpha)).
\end{equation}
Let us consider 

$$
\bar{f}(x)=\left\{\begin{array}{ccc}
f(x) &\mbox{if} & x>\frac{1}{6q_{w-l}};\\
0&\mbox{if}& \text{otherwise.}\end{array}\right.
$$
Hence, by (\ref{xj}) 
\begin{equation}\label{bf}f^{(q_{w-l})}(x)-f(x+i_0\alpha)=\bar{f}^{(q_{w-l})}(x)\text{ and }f^{(q_{w-l})}(y)-f(y+i_0\alpha)=\bar{f}^{(q_{w-l})}(y). 
\end{equation}
By $0\notin [x+i\a,y+i\a]$ for $i=0,...,wq_w-1$ there exists $\theta\in [x,y]$ such that $\bar{f}^{(q_{w-l})}(x)-\bar{f}^{(q_{w-l})}(y)=\|x-y\||\bar{f}'^{(q_{w-l})}(\theta)|$. But as in Lemma \ref{koksi} we get the following:
$$|\bar{f}'^{(q_{w-l})}(\theta)|\leq q_{w-l}(2q_{w-l})^{-\gamma}+ 2|\gamma|(2q_{w-l})^{1-\gamma}+ |\gamma|(6q_{w-l})^{1-\gamma}\leq 46q_{w-l}^{1-\gamma}.$$
Therefore, using (\ref{ela}) and (\ref{bf}), (\ref{del}) and (\ref{quwu}), for every $n\leq i_0$, we have 
$$0<f^{(n)}(x)-f^{(n)}(y)\leq \|x-y\||\bar{f}'^{(q_{w-l})}(\theta)|\leq 46q_{w-l}^{1-\gamma}\frac{2c}{q_w^{1-\gamma}}<\frac{100c}{d}.$$
For $wq_w> n>i_0$, by \eqref{seq} and monotonicity of $f'$, we have  for some $\theta_0\in [x+i_0\alpha,y+i_0\alpha]\stackrel{(\ref{moz}),(\ref{del})}{\subset}[0,\frac{2}{dq_w}]$, and
$$f^{(n)}(x)-f^{(n)}(y)\geq f(x+i_0\alpha)-f(y+i_0\alpha)=\|x-y\||f'(\theta_0)|\geq \frac{1}{q_w^{1-\gamma}}|\gamma|(\frac{dq_w}{2})^{1-\gamma}> \frac{|\gamma|d}{2}.$$
This finishes the proof of \eqref{sho}.

The proof of Lemmata \ref{w0w} is complete. \end{proof}

This finishes the proof of Theorem \ref{aby}.

\appendix

\section{Proof of Proposition \ref{propEE}} \label{App.A} 
\label{EE}
Consider the Gauss map $T:[0,1)\to [0,1)$, $Tx:=\left\{\frac{1}{x}\right\}$, $T(0)=0$, and let $\mu$ be its invariant probability measure given by its density with respect to the Lebesgue measure $ \frac{1}{\log 2} \frac{1}{1+x} dx$.
\begin{lemma}\label{nize} There exists a constant $C>0$ such that for every $a\in \T$ and for every $k\neq l\in \N$
\begin{equation}\label{niezal}\mu\left(T^{-k}((0,a))\cap T^{-l}((0,a))\right)\leq C\mu((0,a))^2.
\end{equation}
\end{lemma}
\begin{proof}  Assume that $l>k$. Then ($\mu$ is $T$-invariant)
$$\mu\left(T^{-k}((0,a))\cap T^{-l}((0,a))\right)=\mu\left((0,a)\cap T^{k-l}(0,a)\right).$$
Note that $T^{k-l+1}(0,a)=\bigcup_{i=1}^{+\infty}(c_i,d_i)$ for some disjoint intervals $(c_i,d_i)$ $i=1,...,+\infty$. We will prove that for every $i\in \N$
\begin{equation} \label{eq0} \mu(T^{-1}(c_i,d_i)\cap (0,a))\leq C\mu((c_i,d_i))\mu((0,a)) \end{equation}
which implies \eqref{niezal} since 
$$\mu\left((0,a)\cap T^{k-l}(0,a)\right)\leq C\mu((0,a))\mu(\bigcup_{i=1}^{+\infty}(c_i,d_i))
=C\mu(T^{k-l+1}(0,a))\mu((0,a))=C\mu((0,a))^2.$$

To prove \eqref{eq0}, note that $T^{-1}(c_i,d_i)=\bigcup_{j=1}^{+\infty}(\frac{1}{d_i+j},\frac{1}{c_i+j})$. It follows that 
$$
\sum_{j=1}^{+\infty} (\frac{1}{d_i+j},\frac{1}{c_i+j})\cap (0,a)\subset \bigcup_{j\geq \frac{1}{a}-d_i}(\frac{1}{d_i+j},\frac{1}{c_i+j}).
$$
Therefore, 
\begin{multline*}
\mu(T^{-1}(c_i,d_i)\cap (0,a))\leq \sum_{j\geq \frac{1}{a}-d_i}\mu((\frac{1}{d_i+j},\frac{1}{c_i+j}))\\ \leq C \sum_{j\geq \frac{1}{a}-d_i} \frac{c_i-d_i}{(c_i+j)(d_i+j)} \leq  C\mu((0,a))\mu((c_i,d_i))
\end{multline*}
for some constant $C>0$ (since the density function $f(x)=\frac{1}{1+x}$ is bounded from above and below on $[0,1]$). This completes the proof. \end{proof}

\begin{proposition} Let $d>0$ and set
$$\mathcal{A}:=\left\{x=[0;a_1,...]\;:\: \exists_{N_0=N_0(x)}\forall_{n\geq N_0} \left|\{k\in [n^2,(n+1)^2]\;:\; a_{k}\geq dk^{\frac{7}{8}}\}\right|<2\right\}.$$
Then $\lambda(\mathcal{A})=1$.
\end{proposition}
\begin{proof}
 We will prove that $\lambda(\mathcal{A}^c)=0$. To do this we will prove that $\mu(\mathcal{A}^c)=0$ ($\lambda$ and $\mu$ are equivalent.). Note that for $k\in \N$ if $x=[0;a_1,...,]$ is the continued fraction of $x$, then $T^k(x)=\frac{1}{a_k+\frac{1}{a_{k+1}+\cdots}}$. Therefore $\mathcal{B}\subset\mathcal{A}$, where  
$$\mathcal{B}:=\left\{x\in \T\;:\: \exists_{N_0=N_0(x)}\forall_{n\geq N_0} \left|\{k\in [n^2,(n+1)^2]\;:\; T^kx\leq \frac{1}{dk^{\frac{7}{8}}}\}\right|<2\right\}.$$
We will prove that $\mu(\mathcal{B}^c)=0$. To do this note that 
\begin{equation}\label{mB}\mathcal{B}^c=\bigcap_{N_0=1}^{+\infty}\left(\bigcup_{n\geq N_0}B_n\right),\end{equation}
where $B_n:=\left\{x\in \T\;:\; \left|\{k\in [n^2,(n+1)^2]\;:\; T^kx\leq \frac{1}{dk^{\frac{7}{8}}}\right|\geq 2\right\}$. Moreover,
\begin{equation}\label{bn}B_n\subset \bigcup_{i_1\neq i_2\in [n^2,(n+1)^2]}B^{n}_{i_1,i_2},\end{equation}
 where $B^n_{i_1,i_2}:=\{x\in \T\;:\; T^{i_1}x,T^{i_2}x\in (0,\frac{1}{dn^{\frac{7}{4}}}]\}$. Let us note that 
$$B^n_{i_1,i_2}=T^{-i_1}((0,\frac{1}{dn^{\frac{7}{4}}}))\cap T^{-i_2}((0,\frac{1}{dn^{\frac{7}{4}}})).$$
By (\ref{niezal}) from Lemma \ref{nize}, we get that $\mu(B^n_{i_1,i_2})\leq  {C}{n^{-\frac{7}{2}}}$. Therefore, using (\ref{bn}) and summing up over all $i_1\neq i_2\in [n^2,(n+1)^2]$, we get that 
$\mu(B_n)\leq C{n^{-\frac{3}{2}}}$.  This and (\ref{mB}) yield
$$\mu(\mathcal{B}^c)=\lim_{N_0\to +\infty}\mu\left(\bigcup_{n\geq N_0}B_n\right)=0.$$
This finishes the proof. \end{proof}

\begin{lemma} Let $\alpha \in \mathcal{A}$. Then $\sum_{s\notin K_\alpha}\frac{1}{\log^{\frac{7}{8}}q_s}<+\infty$.
\end{lemma}
\begin{proof} Let $N_0:=N_0(\alpha)$ be the number resulting from the fact that $\alpha\in \mathcal{A}$. 
We will prove that $\sum_{s\notin K_\alpha, s\geq N_0}\frac{1}{\log^{\frac{7}{8}}(q_s)}<+\infty$.
There exists a constant $d>0$ such that for any $s\in \N$
\begin{equation}\label{exp}\log(q_s)\geq (2d)^{\frac{8}{7}}s
\end{equation}
(indeed, the sequence $(q_s)_{s=1}^{+\infty}$ grows exponentially fast). Let $s\notin K_\alpha$, $s\geq N_0$. Then 
$$a_{s+1}q_s+q_{s-1}=q_{s+1}\geq q_s\log^{\frac{7}{8}}q_s,$$
and therefore, for $s\notin K_\alpha$, by (\ref{exp}) 
\begin{equation}\label{as}a_{s+1}\geq (\ln(q_s))^{\frac{7}{8}}-1\geq ds^{\frac{7}{8}}.
\end{equation}
 Since $\alpha\in \mathcal{A}$, for every $k\geq 1$ in every interval of the form $[(N_0+k)^2,(N_0+k+1)^2]$ there is at most one $s$ such that $a_{s+1}\geq ds^{\frac{7}{8}}\geq (d(N_0+k)^{\frac{7}{8}})$. 
Therefore 
$$\sum_{s\notin K_\alpha,s\geq N_0} \frac{1}{\log^{\frac{7}{8}}q_s}\leq 2d\sum_{s\notin K_\alpha ,s\geq N_0}\frac{1}{s^{\frac{7}{8}}}\leq 2d \sum_{k\geq 1} \frac{1}{(N_0+k)^{2\frac{7}{8}}}=2d\sum_{k\geq 1} \frac{1}{(N_0+k)^{\frac{7}{4}}}<+\infty.$$
This finishes the proof. 
\end{proof}

Hence we proved that $\mathcal{A}\subset \mathcal{E}$ and $\lambda(\mathcal A)=1$,  therefore  $\lambda(\mathcal{E})=1$. \hfill $\square$\\

\end{document}